\documentclass[a4paper]{amsart}
\usepackage{verbatim}
\usepackage[utf8]{inputenc}
\usepackage{eqparbox}
\newcommand{\eqmathbox}[2][eq]{\eqmakebox[#1]{$\displaystyle #2$}}
\usepackage{xcolor} 
\usepackage[
  colorlinks=true,   
  linkcolor=blue,   
  citecolor=blue,    
  urlcolor=blue      
]{hyperref}
\usepackage{todonotes}
\usepackage{bigints}
\usepackage[toc]{appendix}
\usepackage{bm}
\usepackage[T1]{fontenc}

\usepackage{graphicx}
\usepackage{enumerate}
\usepackage{amsmath,amsfonts,amssymb}
\usepackage{color}
\def\loc{\operatorname{loc}}
\usepackage{cite}
\usepackage{ latexsym }
\definecolor{citation}{rgb}{0.11,0.67,0.84}
\definecolor{formula}{rgb}{0.1,0.2,0.6}
\definecolor{url}{rgb}{0.11,0.67,0.84}
\usepackage{pgf,tikz}
\usepackage{mathrsfs}
\usepackage{fancyhdr}
\usepackage{dutchcal}

\usepackage{tikz}
\usepackage{amsmath}

\newcommand{\medint}{-\kern -,375cm\int}

\newcommand{\medintinrigo}{-\kern -,315cm\int}
\makeatletter
\newcommand{\linethrough}{\mathpalette\@thickbar}
\newcommand{\@thickbar}[2]{{#1\mkern0mu\vbox{
    \sbox\z@{$#1#2\mkern-0.5mu$}%
    \dimen@=\dimexpr\ht\tw@-\ht\z@+2\p@\relax 
    \hrule\@height0.5\p@ 
    \vskip\dimen@
    \box\z@}}
}
\makeatother

\newtheorem{theorem}{Theorem}[section]
\newtheorem{lemma}[theorem]{Lemma}

\newtheorem{remark}{Remark}[section] 
\newtheorem{corollary}[theorem]{Corollary}
\newtheorem{definition}[theorem]{Definition}
\theoremstyle{plain}  

\numberwithin{equation}{section}
\usepackage{amsthm}
\usepackage{amsthm}

\usepackage{dsfont}

\newcommand{\reqnomode}{\tagsleft@false}

\usepackage{hyperref}

\def\dx{\,{\rm d}x}

\def\dt{\,{\rm d}t}
\def\dy{\,{\rm d}y}

\def \d{\,{\rm d}}

\def\dist{\,{\rm dist}}

\def\supp{\,{\rm supp}}

\allowdisplaybreaks
\makeatletter
\DeclareRobustCommand*{\bfseries}{%
  \not@math@alphabet\bfseries\mathbf
  \fontseries\bfdefault\selectfont
  \boldmath
}

\makeatother

\newlength{\defbaselineskip}
\newcommand{\mint}{\mathop{\int\hskip -1,05em -\, \!\!\!}\nolimits}

\newcommand{\N}{\mathbb{N}}
\newcommand{\R}{\mathbb{R}}

\newcommand{\mm}{\mathsf{M}}
\newcommand{\nn}{\mathsf{N}}

\newcommand{\rr}{\varrho}

\newcommand{\nr}[1]{\lVert #1 \rVert}

\newcommand{\tx}[1]{\textnormal{\texttt{#1}}}

\def\loc{\operatorname{loc}}

\newcommand{\vp}{\varphi}

\def\eqn#1$$#2$${\begin{equation}\label#1#2\end{equation}}

\def\supp{\,{\rm supp }}

\def\XXint#1#2#3{{\setbox0=\hbox{$#1{#2#3}{\int}$}
     \vcenter{\hbox{$#2#3$}}\kern-.5\wd0}}

\def\XXint#1#2#3{{\setbox0=\hbox{$#1{#2#3}{\int}$ }
		\vcenter{\hbox{$#2#3$ }}\kern-.6\wd0}}

\title{On the Lavrentiev gap for manifold-valued maps}
\author[Antonini]{Carlo Alberto Antonini}  \address{Carlo Alberto Antonini \\ Istituto Nazionale di Alta Matematica ``Francesco Severi'' (INdAM) and
Dipartimento di Matematica e Informatica ``Ulisse Dini'',
Universit\`a di Firenze,
Viale Morgagni 67/A, 50134
Firenze,
Italy\\ ORCID ID: 0000-0002-7663-1090}
\email{\url{antonini@altamatematica.it}}
\author[De Filippis]{Filomena De Filippis}  \address{Filomena De Filippis\\Fachbereich Mathematik, Universität Salzburg, Hel lbrunner Str. 34, 5020 Salzburg, Austria \\ ORCID ID: 0000-0002-2784-1411}
\email{\url{filomena.defilippis@plus.ac.at}}
\author[Pacchiano Camacho]{Cintia Pacchiano Camacho}  \address{Cintia Pacchiano Camacho\\Instituto de Matem\'aticas, Unidad Cuernavaca, Universidad Nacional Aut\'onoma de M\'exico, Av. Universidad, 62210, Cuernavaca, Morelos, Mexico\\
ORCID ID: 0009-0004-6210-4013}
\email{\url{cintia.pacchiano@im.unam.mx}}

\begin{document}



\begin{abstract}
We characterize necessary and sufficient conditions ensuring 
approximation by smooth maps in the Sobolev space $W^{1,\varphi}$ between suitable compact manifolds, where the prototypical example of Young function $\varphi$ is the double-phase integrand.
\end{abstract}

\maketitle
\begin{center}
\begin{minipage}{12cm}
  \small
  \setcounter{tocdepth}{1}
  \tableofcontents
\end{minipage}
\end{center}
\section{Introduction}
\noindent
\\

\noindent This work is devoted to establishing the density of smooth maps between Riemannian manifolds
in nonhomogeneous spaces characterized by the finiteness of certain anisotropic energies. Let
\eqn{M1}
$$ \mm \text{ be an $\tx{m}$-dimensional oriented compact Riemannian manifold} $$
\eqn{N1}
$$ \nn \text{ be an $\tx{n}$-dimensional oriented compact Riemannian manifold without boundary}$$
and consider the Sobolev space $W^{1,\vp}$, where $\vp$ is a Young function. The problem of approximating Sobolev maps between manifolds is classical, and its resolution depends critically on both the function space and the topology of the manifolds. For maps with values in $\R^N$ and Young function $\varphi(t)=t^p$,  the approximation argument is straightforward: standard convolution with a smooth kernel produces a sequence of smooth maps that converge strongly. 
When the target is a manifold $\nn$, the situation is more delicate.   Convolution now takes values in the convex hull of $\nn$, so one must subsequently project these values back onto $\nn$, for instance by using the nearest-point projection.
 This procedure works  in the classical Sobolev setting $\vp(t)=t^p$ for the  superdimensional case $p \ge \tx{m}$. Indeed, by Morrey-Sobolev embedding, maps in $W^{1,p}(\mm,\nn)$ are continuous when $p>\tx{m}$ or belong to the space of functions with vanishing mean oscillation when $p = \tx{m}$ (see \cite{Br}), so projecting them back onto $\nn$ then yields a sequence in $C^\infty(\mm,\nn)$ converging strongly in $W^{1,p}(\mm,\nn)$. The case $p < \tx{m}$ is more subtle, and the possibility of strong approximation depends on the topology of both manifolds. Even for the unit sphere $\tx{S}^2$ the problem is nontrivial. A classical example due to Schoen \& Uhlenbeck \cite{SU} is the map
\begin{displaymath}
u: \tx{B}^3\to \tx{S}^2, \qquad u(x) = \frac{x}{|x|}.
\end{displaymath}
Then, $u \in W^{1,p}(\tx{B}^3,\tx{S}^2)$, for $1 \leq p<3$, but it cannot be strongly approximated in $W^{1,p}(\tx{B}^3, \tx{S}^2)$ by smooth maps with values in $\tx{S}^2$ whenever $2 \le p < 3$.  Seminal work by Bethuel \cite{B} and Hang \& Lin \cite{HL} clarified the approximation problem in the subcritical regime. They showed that for maps in $ W^{1,p}(\mm,\nn)$, the space $C^{\infty}(\mm, \nn)$ is dense if and only if $\mm$ satisfies the $([p]-1)$-extension property with respect to $\nn$ (see \cite[Definition 2.3]{HL}) and the $[p]$-homotopy group of $\nn$ is trivial, where $[p]$ denotes the integer part of $p$. Hajłasz \cite{HJ} later extended these results to more general manifold domains showing that density holds provided $\nn$ is $[{p}]$-connected. The space $W^{1,\tx{m}}(\mm, \nn)$ is thus borderline if one wishes to approximate with smooth maps  avoiding topological restrictions on the manifolds. However, this property persists in slightly larger Sobolev-type spaces, such as Orlicz spaces $W^{1,A}(\mm, \nn)$ built upon Young functions $A$ satisfying suitable conditions depending on $\tx{m}$. The key insight, originating in the work of Hajłasz, Iwaniec, Malý \& Onninen \cite{HIMO}, is that for maps in Sobolev spaces slightly larger than $W^{1,\tx{m}}(\mm,\nn)$, it is possible to detect certain sets on which a given map is still continuous. This leads to the property of {vanishing web oscillations}. The answer in the Orlicz setting was provided by Carozza \& Cianchi \cite{CC}, who showed that if the Young function $A$ satisfies the $\Delta_2$ and $\nabla_2$ conditions near infinity together with a sharp integral condition (which places $W^{1,A}$ either slightly larger than $W^{1,\tx{m}}$ or contained in  $W^{1,\tx{m}}$) then every map in $W^{1,A}(\mm, \nn)$ possesses vanishing web oscillations, establishing the density of smooth maps without any topological constraints on the manifolds. \\

\noindent In the present paper, we go beyond the classical and Orlicz settings, and consider the more general Musielak-Orlicz spaces $W^{1,\vp}(\mm, \nn)$, where the function $\vp$ depends explicitly on the variable $x \in \mm$ as well. We show that, under suitable assumptions on $\varphi$ and appropriate topological conditions on $\mm$ and $\nn$, the space $C^{\infty}(\mm,\nn)$ is dense in the Musielak-Orlicz space $W^{1,\varphi}(\mm,\nn)$. To our knowledge, this is the first systematic study of smooth approximation for manifold-valued maps in this kind of framework. We consider a map $\varphi:\mm \times [0,\infty)  \to [0,\infty)$ such that: 
\eqn{v1}
$$ x \mapsto \vp(x,t) \quad \text{is measurable for all } t \in [0,\infty),$$
and, for every $x \in \mm$,
\eqn{v0}
$$ \varphi(x,t)=0 \quad \text{if and only if} \quad t=0,$$
\eqn{v2}
$$ t \mapsto \varphi(x,t) \quad \text{is convex},$$
\eqn{v3}
$$ \exists  \beta, \gamma > 1: \ t \mapsto  \frac{\varphi(x,t)}{t^\beta}\text{ is almost increasing}, \quad  t \mapsto \frac{\varphi(x,t)}{t^\gamma} \text{ is almost decreasing },$$
\eqn{v3.1}
    $$
    \exists \ \tx{L} \geq 1: \ \tx{L}^{-1} \leq \varphi(x,1) \leq \tx{L}.
    $$
   Moreover, we assume the following: for all $\tx{d}> 1$ there exists a constant $\tx{c}\equiv \tx{c}(\beta,\gamma,\tx{d}) $, where $\beta$ and $\gamma$ are the exponents arising in \eqref{v3}, such that
\eqn{v4}
$$ \varphi(x_1,t) \leq \tx{c}(\varphi(x_2,t)+ 1), \text{ for all } t \in \left[0,\tx{d}\rr^{-\min \left \{1,\frac{\tx{m}}{\beta} \right \}}  \right],$$
for all $x_1,x_2 \in \overline{\mathbb{B}}_{\rr}$. Here, $\mathbb{B}_\rr $ is a geodesic ball of radius $\rr < 1/2$ and $\overline{\mathbb{B}}_\rr$ denotes its closure.
\\

\noindent
These hypotheses are naturally satisfied by a wide class of functionals, most notably the {double phase integrand} 
\eqn{dph}
$$\vp(x,t) = t^p + a(x) t^q,$$ 
with $0 \leq a(\cdot) \in C^{0,\alpha}$, $\alpha \in (0,1]$
and the exponents $1<p<q$ such that 
\eqn{1pq}
$$q \leq p+\alpha \max \left \{1,\frac{p}{\tx{m}} \right \},$$
see \cite[Lemma 3.1]{BGS}. A more general example satisfying our assumptions is provided by the variable-exponent integrand
\eqn{ve}
$$ \vp(x,t)=t^{p(x)} +a(x)t^{q(x)},$$
where $1<p \leq p(x) \leq q(x) \leq q $ and $p(\cdot), q(\cdot)$ are log-H\"older continuous and $0\leq a(\cdot) \in C^{0,\alpha}$, $\alpha \in (0,1]$ and \eqref{1pq} holds, see \cite[Lemma 3.3]{BGS}.\\

\noindent Double phase functionals were introduced by Zhikov \cite{Z1, Z0} in the framework of homogenization and in relation to the Lavrentiev phenomenon. Moreover, this class of variational integrals represents a model for strongly anisotropic materials, with
significant applications in elasticity theory. The interplay between the \(p\)- and \(q\)-growth regimes is dictated by the vanishing behavior of the coefficient \(a(\cdot)\). Over the past decade, an extensive regularity theory has emerged for double phase and, more generally, for nonuniformly elliptic functionals. In this direction, Marcellini \cite{ma2, ma4, ma1, ma5} investigated $(p,q)$-nonuniformly elliptic integrals, providing the first systematic results in this setting. Later contributions by Baroni, Colombo, and Mingione \cite{BCM, comi2, comi1} provided a complete framework for regularity, which subsequently inspired a vast body of literature, see \cite{bb901, G5, HO1,HO2,MNP,NC,jss, DFM, F1.1, F3.1, K2, DNP} and the references therein. Let us also mention that problems characterized by linear or nearly linear growth have been analyzed, among others, in \cite{BG, G2, G3}. \\

\noindent Now, condition \eqref{1pq} is connected with the nonuniform ellipticity of the associated Euler-Lagrange system
\begin{equation}\label{el}
-\mathrm{div}\,A(x,Du) = 0,
\end{equation}
where
\begin{equation*}
A(x,z) = |z|^{p-2}z + (q/p)\, a(x)|z|^{q-2}z.
\end{equation*}
As shown in \cite[Section 3]{BGS}, for \eqref{dph} and \eqref{ve}, it can be verified that assumption \eqref{v4} holds if \eqref{1pq} is true. Observe that the (nonlocal) ellipticity ratio
\begin{equation}\label{ratio}
\displaystyle 
\frac{\sup_{x\in \tx{B}} \text{highest eigenvalue of }\partial_z A(x, Du)} 
{\inf_{x\in \tx{B}} \text{lowest eigenvalue of }\partial_z A(x, Du)}
\,\approx\, 
1 + \|a\|_{L^{\infty}(\tx{B})} |Du|^{q-p}
\end{equation}
can become unbounded on any ball \(\tx{B} \subset \Omega\) intersecting the transition region \(\{ a(\cdot)=0\}\). Then, condition \eqref{1pq}, is precisely designed to ensure that \(a(\cdot)\) decays sufficiently rapidly to prevent uncontrolled degeneracy, thereby guaranteeing maximal gradient regularity for local minimizers. The necessity of such a condition has been confirmed through explicit counterexamples in \cite{BDS, Bal23, ELM, FMM, BS} showing that, when \eqref{1pq} is violated, the so-called {Lavrentiev phenomenon} can occur, i.e., the infimum of a variational integral over smooth functions is strictly greater than its infimum over the natural energy space, and this prevents the density of smooth functions in $W^{1,\vp}$. We remark that, in our manifold-valued framework, the compactness of $\nn$ ensures the boundedness of maps, which allows us to take $q \leq p+\alpha$, in the case $p \leq \tx{m}$.

We also recall that, if \( w \notin L^{\infty} \), the appropriate bound becomes 
\( q \leq p + p\alpha/\tx{m} \), 
corresponding to \eqref{v4} with 
\( t \in [0, \tx{d}\rr^{-\tx{m}/p}] \), see \cite{comi2}. A related but more delicate situation arises in the study of {strongly nonuniformly elliptic functionals}, where even the pointwise ellipticity ratio becomes unbounded, as \( |z| \to +\infty \), in a nonbalanced fashion:
\begin{equation*}
\displaystyle 
\frac{\text{highest eigenvalue of }\partial_z A(x, Du)} 
{\text{lowest eigenvalue of }\partial_z A(x, Du)}
\,\approx\, 
1 + |z|^{\delta}, 
\quad \delta >0.
\end{equation*}
Even though these systems are strongly nonuniformly elliptic, recent work has developed a Schauder’s regularity theory \cite{DefS1, DefS3}, which also applies to variational problems with nearly linear growth, see  \cite{DefS2, DFP, DFDFP}.\\ \\
\noindent
Our goal is to provide a comprehensive account of smooth approximation in $W^{1,\vp}(\mm, \nn)$, for $\varphi$ satisfying \eqref{v1}-\eqref{v4}. The first result establishes that if the function $\varphi$ satisfies suitable integral conditions allowing $W^{1,\vp}$ to be either slightly larger than $W^{1,\tx{m}}$ or contained in it, then maps in $W^{1,\varphi}(\mathsf{M},\mathsf{N})$ have vanishing web oscillations, and thus smooth maps are dense in the related Musielak-Orlicz Sobolev space. Here,
no topological assumption on the manifolds other than \eqref{M1},\eqref{N1} are assumed.

\begin{theorem} \label{theorem1}
    Let $\mm$ and $\nn$ be Riemannian manifolds  as in  \eqref{M1}, \eqref{N1}.  Let $\vp: \mm \times [0,\infty) \to [0,\infty)$ be a function satisfying \eqref{v1}-\eqref{v4} and denote by
    \begin{equation}\label{def:inf}
        \varphi^-_{\mm}(t) := \displaystyle\inf_{x \in \mm}\varphi(x,t).
    \end{equation}
 Assume either
  \eqn{1a}  $$ \tx{m}=2 \text{ and } \int^{\infty} \frac{\varphi^-_{\mm}(t)}{t^3} \dt  = \infty, $$
  or
   \eqn{1b} $$ \tx{m} \geq 3 \text{ and } \int^{\infty} \left (\frac{t}{\varphi^{-}_{\mm}(t)} \right )^{\frac{2}{\tx{m}-2}} \left ( \int_{t}^{\infty} \left (\frac{s}{\varphi^-_{\mm}(s)} \right )^{\frac{1}{\tx{m}-2}} \d s\right )^{-\tx{m}} \dt = \infty.$$
    Then, $C^{\infty}(\mm,\nn)$ is dense in $W^{1,\varphi}(\mm,\nn)$.
\end{theorem}
\noindent
Assumptions \eqref{1a} and \eqref{1b} are designed to control the admissible growth rates in the $t$-variable. In particular, they allow for functions whose infimum grows slightly slower than $t^\tx{m}$ at infinity. A typical example in the double-phase setting is
\[
\varphi(x,t) = \frac{t^\tx{m}}{\log t} + a(x)\, t^\tx{m},
\]
with $a(\cdot) \ge 0$ Hölder continuous function. More general examples allowed by Theorem \ref{theorem1} are double phase functionals of the type
\[
\varphi(x,t) = \frac{t^p}{\log t} + a(x)\, t^q\quad\text{or}\quad  \varphi(x,t) = t^p + a(x)\, t^q
\]
with $p\geq \tx{m}$, $0\leq a(\cdot)\in C^{0,\alpha}$, and $q$ satisfying \eqref{1pq}.
Variable exponent functionals of the type \eqref{ve}  are also admitted, as long as $\tx{m}\leq p(x)\leq q(x)\leq q$, $0\leq a(\cdot)\in C^{0,\alpha}$, $p(\cdot),q(\cdot)$ are log-H\"older continuous and \eqref{1pq} is satisfied.\\ \\
\noindent On the other hand, the next result shows that, under a suitable topological condition on the target manifold, the density of smooth maps still holds without imposing any relation with the dimension $\tx{m}$ of the domain manifold. Here, we assume that
\eqn{N2}
$$ \nn \text{ is } \tx{k}\text{-connected},$$
with an interplay between $\tx{k}$ and the exponent $\gamma$ from \eqref{v3}. Note that $\tx{k} <  \tx{n}$, otherwise $\nn$ would be contractible.
\begin{theorem} \label{theorem2}
    Let $\mm$ and $\nn$ be Riemannian manifolds  satisfying  \eqref{M1}, \eqref{N1} and \eqref{N2}. Assume that $\vp: \mm \times [0,\infty) \to [0,\infty)$ is a function satisfying \eqref{v1}-$(\ref{v3})_1$, $(\ref{v3})_2$ with
    \begin{equation}\label{cond:gamma}
        \gamma\in (1,\tx{k}+1],
    \end{equation}
 and \eqref{v3.1}, \eqref{v4}. Then, if $\gamma < \tx{k}+1$, $C^{\infty}(\mm,\nn)$ is strongly dense in $W^{1,\varphi}(\mm,\nn)$; if $\gamma=\tx{k}+1$, $C^{\infty}(\mm,\nn)$ is weakly dense in $W^{1,\varphi}(\mm,\nn)$.
\end{theorem}
\noindent  
The prototypical example of target manifold fulfilling the assumptions of Theorem \ref{theorem2} is $\nn=\tx{S}^{N-1}$, namely, the $(N-1)$-dimensional sphere in $\R^N$,   which is $(N-2)$-connected.
\\

\noindent
Before proceeding, we make a few observations regarding Theorem \ref{theorem2}, emphasizing its connection with the classical Sobolev case $W^{1,p}$ and the role of topological assumptions.
\begin{remark}
\normalfont
    We observe that the situation of Theorem \ref{theorem2} reflects the one occurring in the classical case $\varphi=t^p$ where, for $p < \tx{m}$, some topological assumptions on $\mm$ and  $\nn$ are necessary for the density result, as shown, for instance, in \cite{B,HJ, BPV, HL, HL1, HL2, Det}. Indeed, in Theorem \ref{theorem2}, $W^{1,\vp}$ can be a function space arbitrarily larger than  $W^{1,\tx{m}}$, so Theorem \ref{theorem2} complements Theorem \ref{theorem1}.
\end{remark}

\noindent
We observe that a direct corollary of the above results is the absence of Lavrentiev phenomenon in any of the two settings addressed.
\begin{corollary} \label{corollary}
Suppose that the hypotheses of either Theorem~\ref{theorem1} or Theorem~\ref{theorem2} are satisfied. Then
    $$ \inf_{W^{1,\vp}(\mm,\nn)} \int_\mm \vp(x,|Du|) \d \mathscr{H}^{\tx{m}} = \inf_{C^\infty(\mm,\nn)} \int_\mm \vp(x,|Du|) \d \mathscr{H}^{\tx{m}}.$$
\end{corollary}

\noindent
We end the analysis by explaining that, in this nonautonomous setting where $\varphi$ is described by structural conditions \eqref{v1}-\eqref{v3.1}, the local assumption \eqref{v4} plays a fundamental role. Indeed, when this condition fails, genuine counterexamples arise, even when $W^{1,\vp}$ is a function space smaller than  $W^{1,\tx{m}}$ or when the target manifold is $\tx{k}$-connected; a phenomenon that does not occur in the classical $W^{1,p}$ framework, nor in the $W^{1,A}$ setting considered in \cite{CC}. We focus on the double phase functional \eqref{dph}, and we express failure of
condition \eqref{v4} as 
\eqn{pq}
$$q > p + \alpha \max \left \{ 1,\frac{p-1}{\tx{m}-1} \right \}.$$ 
This condition is sharp for both $p \leq \tx{m}$ and $p >\tx{m}$ 
as highlighted in \cite[Remark 35]{BDS}. Following the fractal constructions of Balci, Diening \& Surnachev \cite{BDS}, we show that when \eqref{pq} holds, Lavrentiev phenomenon occurs and consequently smooth maps are not dense in $W^{1,\vp}(\mm, \nn)$, for $\mm :=[-1,1]^\tx{m} $, i.e. the $\tx{m}$-dimensional cube, and $\nn=\tx{S}^{N-1}_{{\Lambda}}$, where $\Lambda$ represents a suitable radius of the sphere; note that the sphere $\tx{S}^{N-1}_{{\Lambda}}$ is $(N-2)$-connected. This provides the first vectorial counterexample of the Lavrentiev phenomenon, which we further extend to manifold-valued maps. This is the content of the next theorem.
First, let us introduce some notation. We set $Q_\mm=(-1,1)^{\tx{m}}$, and for a given function $u_0\in C^\infty(\overline{Q}_\mm,\tx{S}^{N-1}_\Lambda)$, we define the spaces
\begin{equation}
C^\infty_{u_0}(\overline{Q}_\mm,\tx{S}^{N-1}_\Lambda):=\big\{v\in C^\infty(\overline{Q}_\mm,\tx{S}^{N-1}_\Lambda): \,v=u_0\,\text{on $\partial Q_\mm$}  \big\},   
\end{equation}
and correspondingly
\begin{equation}
    W^{1,\vp}_{u_0}(Q_\mm,\tx{S}^{N-1}_\Lambda):=\big\{v\in  W^{1,\vp}(Q_\mm,\tx{S}^{N-1}_\Lambda): \,v=u_0\,\text{on $\partial Q_\mm$}  \big\}.
\end{equation}
We have the following:
\begin{theorem} \label{tcontr}
 Let $\varphi$ be as in \eqref{dph} and assume that $q,p>1$ and $\alpha\in (0,1]$ are numbers such that \eqref{pq} holds. Then there exist $\Lambda \geq 1$ and functions $a \in C^{\alpha}(\overline{Q}_\mm)$, $u_0 \in C^{\infty}(\overline{Q}_\mm,\tx{S}^{N-1}_{{\Lambda}})$, $\bar{u} \in W^{1,\vp}_{u_0}(Q_\mm, \tx{S}^{N-1}_{{\Lambda}})$ such that
 \eqn{lavrentiev}
$$ \inf_{u \in W^{1,\vp}_{u_0}(Q_\mm,\tx{S}^{N-1}_\Lambda)} \int_{Q_\mm} \varphi(x,|Du|) \d x < \inf_{u\in C^{\infty}_{u_0}(\overline{Q}_\mm,\tx{S}^{N-1}_\Lambda)} \int_{Q_\mm} \varphi(x,|Du|) \d x$$
and there is no sequence $\{u_{\ell}\}_{\ell} \subset C^{\infty}_{u_0}(\overline{Q}_\mm,\tx{S}^{N-1}_{{\Lambda}})$ such that $\int_{Q_\mm} \vp(x,|Du_{\ell} -D\bar{u}|) \d x \to 0$.
\end{theorem}

\noindent
The structure of the paper is as follows. Section \ref{sec2} introduces the notation, functional setting, and auxiliary results required for our analysis. Section \ref{sec5} contains the proofs of our main approximation results, namely Theorems \ref{theorem1} and \ref{theorem2}, establishing also the absence of the Lavrentiev phenomenon. Finally, Section \ref{sec4} provides the explicit counterexample that highlights the sharpness of assumptions \eqref{v4}.

\section{Preliminaries} \label{sec2}
\noindent
This section introduces the basic notation, the main properties of the Young function we consider, and the Musielak-Orlicz space we work with. We also collect several useful results and technical lemmas, which follow from \cite{CC}, and will be employed throughout the paper.

\subsection{Notation} \label{notation}
Here we fix the notation used in the sequel. We denote by $\mm \subset \R^d$ and $\nn \subset \R^N$ two manifolds without boundary as in \eqref{M1} and \eqref{N1}. For $\tilde{x} \in \mm$ and $r>0$, we write $\mathbb{B}_r \equiv \mathbb{B}_r(\tilde{x})$ and $\mathbb{S}^{\tx{m}-1}_r(\tilde{x}) \equiv \mathbb{S}^{\tx{m}-1}_r$ to denote respectively the geodesic ball and sphere of radius $r$ centered at $\tilde{x}$. Likewise, $\tx{B}_r \equiv \tx{B}_r(\bar{x})$ and $\tx{S}^{\tx{m}-1}_r(\bar{x}) \equiv \tx{S}^{\tx{m}-1}_r$ stand for the Euclidean ball and sphere of radius $r$ centered at $\bar{x} \in \R^{\tx{m}}$. Unless stated otherwise, all balls and spheres are assumed to share the same center. The symbol $\nabla_\mathbb{S}$ stands for the gradient on $\mathbb{S}^{\tx{m}-1}_r$. We use $c$ to denote a generic positive constant, which can change from line to line; whenever relevant, its dependencies will be explicitly specified. In the estimates below, any dependence on geometric features $\nn$, such as the $L^\infty$-norm of maps into $\nn$ (note that $\nn$-valued maps have finite $L^\infty$-norm because of the compactness of $\nn$), will be indicated simply as $c(\nn)$. Similarly, dependencies on geometric features of $\mm$ will be denoted by $c(\mm)$. With  \(R_\mm\) we denote a positive number such that for any point \(x\) in \(\mm\), all geodesic spheres centered at \(x\) with a radius smaller than \(R_\mm\) are contained in some coordinate chart from a reference atlas. Furthermore, for any \(x\), a system of geodesic spherical coordinates centered at $x\in \mm$ is well defined for $r \in (0,R_\mm)$. We use symbols "$\lesssim$" with subscripts, to indicate that a certain inequality holds up to constants whose relevant dependencies are marked in the suffix.  We conclude by introducing the definition of $j$-connectedness for manifolds, which will play an important role in the proof of Theorem \ref{theorem2}.
\begin{definition}[$j$-connected manifolds]\label{j connected}
    Given an integer $j \geq 0$, a manifold $\tilde{\mm}$ is said to be $j$-connected if its
first $j$ homotopy groups are trivial, that is $\pi_0(\tilde{\mm}) = \pi_1(\tilde{\mm}) = \dots= \pi_{j-1}(\tilde{\mm})=\pi_j(\tilde{\mm})= 0$.
\end{definition}
\subsection{Properties of Young functions} \label{2.2}
In this section we recall the notion of a Young function and present several properties of 
$\vp$ defining our Musielak-Orlicz space. A map \( A: [0, \infty) \to [0, \infty) \) is  a \textit{Young function}  if it is convex, non constant and vanishes only at $0$.

Accordingly, a mapping \( \varphi: \mm \times [0, \infty) \to [0, \infty) \) is called a {generalized Young function} provided it satisfies the following two conditions:
\begin{itemize}
    \item For every \( x \in \mm \), the function \( t \mapsto \varphi(x, t) \) is a Young function,
    \item For every \( t \ge 0 \), the function \( x \mapsto \varphi(x, t) \) is measurable.
\end{itemize}
\noindent Next, the Young conjugate $A^* \colon [0,\infty) \to [0,\infty)$  of $A$ is given by
\begin{equation}\label{young:conj}
    A^*(\eta) = \sup_{\xi \geq 0} \big(\xi \eta - A(\xi)\big).
\end{equation}  
The second convex conjugate $A^{**}$ is defined as
\begin{equation}\label{sec:yconj}
    A^{**}(\xi) = \sup_{\eta \geq 0} \big(\xi \eta - A^*(\eta)\big).
\end{equation} 
It is also well known that $A^{**}(\cdot)$ is the greatest convex minorant of $A(\cdot)$. Analogously, the Young conjugate function $\vp^*: \mm \times [0,\infty) \to [0,\infty)$ of $\vp(x,t)$ is the Young function defined by
$${\vp}^*(x,t)= \sup_{s \geq 0}\{st-\vp(x,s) \}, \quad \text{for every } x\in \mm .$$

\noindent For what concerns the growth conditions, we say that $b: [0,\infty) \to [0,\infty)$ is \textit{almost increasing} or \textit{almost decreasing} if there exists $c \geq 1$ such 
that for any $0 < t < s < \infty$, it holds $b(t)\leq cb(s)$ and $b(s) \leq cb(t)$, respectively. \\

\noindent Note that, for a function $\vp$ satisfying the almost monotonicity $(\ref{v3})_1$ with a constant $c \geq 1$, for all $0<m_1 \leq 1  \leq m_2 < \infty$, it holds
\eqn{du1}
$$ \vp(x,m_1t) \leq c\,{m_1}^{\beta}\vp(x,t) \quad \text{and} \quad c^{-1} m_2^\beta \vp(x,t) \leq \vp(x,m_2 t), \quad \text{for all } (x,t) \in \mm \times [0,\infty).$$
Analogously, if $\vp$ satisfies $(\ref{v3})_2$ with a constant $c \geq 1$, then
\eqn{du}
$$ c^{-1} m_1^\gamma \vp(x,t) \leq \vp(x,m_1 t) \quad \text{and} \quad \vp(x,m_2 t) \leq c m_2^{\gamma} \vp(x,t), \quad \text{for all } (x,t) \in \mm \times [0,\infty).$$
Thanks to the monotonicity of $t\mapsto \vp(x,t)$ and \eqref{du1}, it  is immediate to verify that $\vp$ satisfies the sub-additivity property
\begin{equation}\label{subadd}
    \vp(x,t+s)\leq c\,\big(\vp(x,t)+\vp(x,s)\big),\quad\text{for all $x\in \mm$  and for all $t,s\geq 0$,}
\end{equation}
for some constant $c\geq 1$. Moreover, by \eqref{v3}-\eqref{v3.1}, we have that
\begin{equation}\label{control:vp}
  c\,\min\{t^\beta,t^\gamma\}  \leq \vp(x,t)\leq C\,\max\{t^\beta,t^\gamma\}, \quad \text{for all } (x,t) \in \mm \times [0,\infty),
\end{equation}
for some constants $0 <c\leq 1 \leq C$. Next, we say that $\vp$ satisfies the so-called $\Delta_2$-condition if there exists a constant $c$ such that
\eqn{d1}
$$ \vp(x,2t) \leq c\vp(x,t),  \quad \text{for all } (x,t) \in \mm \times [0,\infty),$$
and that $\varphi$ satisfies the $\nabla_2$-condition if its Young conjugate $\varphi^*$ fulfills the $\Delta_2$-condition. Observe that, by \eqref{du1}-\eqref{du}, we have that
\eqn{d2}
$$ \text{if $\vp$ satisfies \eqref{v3}, then it satisfies both $\Delta_2$- and $\nabla_2$-condition.}$$ Whenever $\vp$ satisfies the $\Delta_2$ or $\nabla_2$ condition, we write $\vp\in \Delta_2$ and $\vp\in \nabla_2$ respectively.\\

\noindent Now, let $\vp_{\mm}^{-}$ be given by \eqref{def:inf}. Since it is an infimum of convex (hence continuous) functions, then $\vp_\mm^-(\cdot)$ is lower-semicontinuous. Moreover, by taking the infimum over $x\in \mm$ in \eqref{v3}, we infer that 
\begin{equation}\label{inf:monot}
      t \mapsto  \frac{\vp_\mm^-(t)}{t^\beta}\text{ is almost increasing},\quad  t \mapsto  \frac{\vp_\mm^-(t)}{t^\gamma} \text{ is almost decreasing.}
\end{equation}
Let us now denote  by $\Psi:[0,\infty)\to [0,\infty)$ the greatest convex minorant of $\vp_\mm^-$, that  is $\Psi=(\vp_\mm^-)^{**}$.
We claim that there exists $C\geq 1$ such that
\begin{equation}\label{equiv}
    \Psi(t)\leq \vp_\mm^-(t)\leq C\,\Psi(2t)\quad\text{for all $t\geq 0$.}
\end{equation}
The left inequality follows from the definition of $\Psi$. Next observe that from \eqref{inf:monot}, we get that $t \mapsto \vp_{\mm}^{-}(t)/t$ is almost increasing, that is  $\vp_{\mm}^{-}(t) \geq c\, \frac{t}{s}\vp_{\mm}^{-}(s)$ for $t >s$, for some constant $c \in (0,1)$. It immediately follows that, for $s>0$ fixed, the function $t \mapsto c\,\left( \frac{t}{s}-1\right ) \vp_{\mm}^{-}(s)$ is a convex minorant of $\vp_{\mm}^{-}$ on $[0,\infty)$, hence $\Psi(t) \geq c\,\left (\frac{t}{s}-1 \right )\vp_{\mm}^{-}(s)$. Thereby choosing $s=t/2$, we obtain \eqref{equiv} with $C=c^{-1}$. 

The advantage of working with $\Psi$ in place of $\vp_\mm^-$ is that $\Psi$ is a convex function. Additionally, thanks to \eqref{equiv}, $\Psi$ satisfies the same monotonicity properties \eqref{inf:monot} of $\vp_\mm^-$, i.e.
\begin{equation}\label{inf:monot1}
      t \mapsto  \frac{\Psi(t)}{t^\beta}\text{ is almost increasing, }\quad  t \mapsto  \frac{\Psi(t)}{t^\gamma} \text{ is almost decreasing.}
\end{equation}
In particular,  $\Psi$ is a Young function satisfying the $\Delta_2$- and $\nabla_2$- conditions. Furthermore, whenever \eqref{1a} and \eqref{1b} are valid, they also hold true with $\Psi$ in place of $\vp_\mm^-$ thanks to \eqref{equiv}, that is
  \eqn{1aa}  $$ \tx{m}=2 \text{ and } \int^{\infty} \frac{\Psi(t)}{t^3} \dt  = \infty, $$
  or
   \eqn{1bb} $$ \tx{m} \geq 3 \text{ and } \int^{\infty} \left (\frac{t}{\Psi(t)} \right )^{\frac{2}{\tx{m}-2}} \left ( \int_{t}^{\infty} \left (\frac{s}{\Psi(s)} \right )^{\frac{1}{\tx{m}-2}} \d s\right )^{-\tx{m}} \dt = \infty.$$
For such $\Psi$, we also define the auxiliary Young function  $\Psi_{\tx{m}-1}$ as
\begin{equation}\label{def:psim}
\Psi_{\tx{m}-1}(t) =
\begin{cases}
\Psi(t) & \text{if } \tx{m} = 2, \\[4pt]
\displaystyle \left( t^{\frac{\tx{m}-1}{\tx{m}-2}} \int_t^\infty \frac{{\Psi}^*(r)}{r^{1 + \frac{\tx{m}-1}{\tx{m}-2}}} \, \d r \right)^{*} & \text{if } \tx{m} \geq 3,
\end{cases}
\end{equation}
for $t \in [0,\infty)$. Here, $(\cdot)^{*}$ denotes the Young conjugate as in \eqref{young:conj}. 

\noindent
\subsection{Musielak-Orlicz spaces}
Here we introduce the Musielak-Orlicz space and provide the basic definitions and properties that will be used throughout the paper. We start considering an open subset $\Omega_\mm \subseteq \mm$, with $\mm$ as in \eqref{M1},  and a Young function $\varphi$. We denote by $L^{0}(\Omega_\mm,\R^N)$ the set of measurable functions on $\Omega_\mm$, equipped with the measure introduced by the Riemannian metric on $\mm$, which agrees with the Hausdorff measure $\mathscr{H}^\tx{m}$.
The Musielak-Orlicz space $L^{\varphi}(\Omega_\mm,\R^N)$ is defined as
$$ L^{\varphi}(\Omega_\mm,\R^N):= \left \{ u \in L^{0}(\Omega_\mm,\R^N) : \int_{\Omega_\mm} \vp \left (x, \frac{|u(x)|}{\lambda} \right ) \d\mathscr{H}^{\tx{m}} < \infty, \text{ for some } \lambda > 0\right \},$$
and we endow such space with the so-called Luxemburg norm
$$ \nr{u}_{L^\vp(\Omega_\mm)}:= \inf \left \{ \lambda > 0: \int_{\Omega_\mm} \vp \left (x,\frac{|u(x)|}{\lambda} \right ) \d \mathscr{H}^{\tx{m}}\leq 1\right \}.$$
Moreover, if $\vp(x, \cdot) \in\Delta_2$, then
$$ L^{\varphi}(\Omega_\mm,\R^N):= \left \{ u \in L^{0}(\Omega_\mm,\R^N) : \int_{\Omega_\mm} \vp \left (x, k|u(x)| \right ) \d\mathscr{H}^{\tx{m}} < \infty, \text{ for all } k > 0\right \}$$
and given a sequence $\{u_{\ell}\}_\ell \in L^\vp(\Omega_\mm,\R^N)$ and $u \in L^\vp(\Omega_\mm,\R^N)$, there holds the convergence equivalence (see \cite[Lemma 3.3.1]{HHbook})
\eqn{conv2}
$$ \lim_{\ell\to \infty}\nr{u_\ell -u}_{L^\vp(\Omega_\mm)}= 0 \ \text{ if and only if } \ \lim_{{\ell} \to \infty} \int_{\Omega_\mm} \vp(x,|u_{\ell}-u|) \d\mathscr{H}^{\tx{m}} =0.$$
We next introduce the corresponding Musielak-Orlicz Sobolev space
$$ W^{1,\vp}(\Omega_\mm, \R^N):= \left \{ u \in W^{1,1}(\Omega_\mm,\R^N) : |Du| \in L^\vp(\Omega_\mm, \R^N)\right \},$$
which is endowed with the norm
$$ \nr{u}_{W^{1,\vp}(\Omega_\mm)} := \nr{u}_{L^{1}(\Omega_\mm)} + \nr{Du}_{L^\vp(\Omega_\mm)}.$$
To describe functions vanishing outside $\Omega_\mm$, we define the space
$$ W^{1,\vp}_0(\Omega_\mm,\R^N):= \left \{  u \in W^{1,\vp}(\Omega_\mm,\R^N): \text{ the extension to $\mm \setminus \Omega_\mm$ of $u$ by $0$ belongs to } W^{1,\vp}(\mm,\R^N)\right \}.$$
Finally, when $\nn$ is a submanifold of $\R^N$ as in \eqref{N1}, we set
$$ W^{1,\vp}(\Omega_\mm,\nn):= \left \{ u \in W^{1,\vp}(\Omega_\mm,\R^N): \text{Im}(u) \subset \nn \text{ holds almost everywhere} \right \}.$$

\subsection{The double of the manifold \texorpdfstring{$\mm$}{M}.} Let $\mm$ be an oriented, compact Riemannian $\tx{m}$-manifold with smooth boundary $\partial \mm\neq \emptyset$. We can consider its double denoted by $\mathcal{D}(\mm)$ as follows: we consider two copies of $\mm$, denoted by $\mm_1,\mm_2$ respectively, and we attach them along the boundary, so that $\partial\mm_1=\partial\mm_2=\partial \mm$. The precise construction can be found in \cite[Theorem 9.29, Example 9.32]{Lee}. 

In particular, there exist two diffeomorphisms $\mathcal{i
}_1:\mm\to \mm_1$, $\mathcal{i_2}:\mm\to \mm_2$ (often called \textit{identification maps} between $\mm$ and its $i$-th copy $\mm_i$), and we have

\begin{equation}\label{double1}
    \mathcal{D}(\mm)=\mm_1\cup \mm_2,\quad \text{with}\quad  \mm_1\cap \mm_2= \mathcal{i}_1(\partial\mm)=\partial \mm_1 =\mathcal{i}_2(\partial\mm)=\partial \mm_2=\partial \mm
\end{equation}
where the last equality is just formal and has to be intended in the sense of diffeomorphisms.
This means that $\mathcal{D}(\mm)$ is made of the two copies of $\mm$, whose interiors are disjoint, and such copies intersect only at the hypersurface $\partial\mm$. Hence, by identifying $\mm$ with one of its copies, say $\mm_1$, we may effectively consider $\mm\subset \mathcal{D}(\mm)$ as a submanifold.

Moreover, if $\mm$ is a compact, oriented Riemannian $\tx{m}$-dimensional manifold, so is the double $\mathcal{D}(\mm)$. The key feature is that
\begin{equation}\label{double}
    \partial\mathcal{D}(\mm)=\emptyset,\quad \text{and thus}\quad \partial\mm\subset \mathcal{D}(\mm)=\mathrm{Int}\,\mathcal{D}(\mm)
\end{equation}
where $\mathrm{Int}$ denotes the interior of the manifold. Therefore, we  have that $\partial\mm$ is a $(\tx{m}-1)$-dimensional submanifold embedded into $\mathcal{D}(\mm)$. 

Let us now consider a $\mathscr{H}^{\tx{m}}$-measurable function $v:\mm\to \R^N$, $N\in \N$; its extension to the double $\mathcal{D}(\mm)$ is immediately defined by setting
\begin{equation}\label{double:ext}
    \tilde{v}:\mathcal{D}(\mm)\to \R^N,\quad \tilde v|_{\mm_1}:=v\circ \mathcal{i}_1,\quad \tilde v|_{\mm_2}:=v\circ \mathcal{i}_2.
\end{equation}
so that  $\tilde{v}|_\mm=v$, modulo identifying $\mm$ with $\mm_i$.  It immediately follows by definition that the images coindice $\mathrm{Im}(v)=\mathrm{Im}(\tilde v)$, and if $v$ is real-valued
\begin{equation}\label{sup:double}
    \sup_{\mm} v=\sup_{\mathcal{D}(\mm)} \tilde v,\quad  \inf_{\mm} v=\inf_{\mathcal{D}(\mm)} \tilde v\quad\text{and hence}\quad \operatorname*{osc}_\mm v=\operatorname*{osc}_{\mathcal{D}(\mm)}\tilde v.
\end{equation}
Moreover, if $v\in W^{1,1}(\mm,\R^N)$, then its extension $\tilde{v}\in W^{1,1}(\mathcal{D}(\mm),\R^N)$. To see this, first recall the following basic fact from the theory of Sobolev spaces: given two functions $w_1\in W^{1,1}(\Omega_1),w_2\in W^{1,1}(\Omega_2)$ on two regular domains sets $\Omega_1,\Omega_2\subset \mathbb{R}^{\tx{m}}$, such that $\Omega_1\cap \Omega_2=\Gamma$ is a regular hypersurface, then the function
\begin{equation*}
w(x)=
    \begin{cases}
        w_1(x)\quad x\in \Omega_1
        \\
        w_2(x)\quad x\in \Omega_2
    \end{cases}
\end{equation*}
belongs to the Sobolev spaces $W^{1,1}(\Omega_1\cup \Omega_2)$ if and only if the traces $w_1|_\Gamma=w_2|_{\Gamma}$ coincide.

Hence, if $v\in W^{1,1}(\mm,\R^N)$, then by \eqref{double:ext} we have $v_i:=\tilde{v}|_{\mm_i}\in W^{1,1}(\mm_i,\R^N)$ for each $i=1,2$, and $v_1=v_2$ on $\mm_1\cap \mm_2=\partial \mm$, so by \eqref{double} and from the preceding discussion (applied on local charts of $\mathcal{D}(\mm)$) we have $\tilde{v}\in W^{1,1}(\mathcal{D}(\mm),\R^N)$.  

Now, let  $\varphi:\mm\times [0,\infty)\to [0,\infty)$ be a Young function as in \eqref{v1}-\eqref{v4}, and define its extension to the double
\begin{equation}\label{double:young}
    \tilde{\varphi}:\mathcal{D}(\mm)\times [0,\infty)\to [0,\infty),\quad \tilde \varphi(\cdot,t)\text{ is the extension \eqref{double:ext} of $\varphi(\cdot,t)$ to $\mathcal{D}(\mm)$  for every $t\geq 0$.}
\end{equation}
It immediately follows from \eqref{double:ext}-\eqref{sup:double} that $\tilde\varphi$ still satisfies \eqref{v1}-\eqref{v4}, and 
\begin{equation}\label{same:inf}
    \tilde{\varphi}_{\mathcal{D}(\mm)}^-(t)=\inf_{x\in \mathcal{D}(\mm)} \varphi(x,t)=\varphi^-_\mm(t)\quad \text{for every $t\geq 0$.}
\end{equation}

Next, let $v\in W^{1,\varphi}(\mm,\R^N)$, and let $\tilde v$ and $\tilde\varphi$ be as in \eqref{double:ext}-\eqref{double:young}. We claim that $\tilde v\in W^{1,\tilde \varphi}(\mathcal{D}(\mm),\R^N) $. Indeed, we saw that $\tilde v\in W^{1,1}(\mathcal{D}(\mm),\R^N)$, and using that $\mm=\mm_j$ (modulo diffeomorphism $\mathcal{i}_j$, $j=1,2$), together with  \eqref{control:vp} and  \eqref{double1}, we deduce
\begin{equation*}
    \int_{\mathcal{D}(\mm)} \tilde{\varphi}(x,|D\tilde v|)\d\mathscr{H}^{\tx{m}}=\int_{\mm_1} \tilde{\varphi}(x,|D\tilde v|)\d\mathscr{H}^{\tx{m}}+\int_{\mm_2} \tilde{\varphi}(x,|D\tilde v|)\d\mathscr{H}^{\tx{m}}\leq C\,\int_{\mm}\varphi(x,|Dv|)\d\mathscr{H}^{\tx{m}}<\infty
\end{equation*}
which shows that $\tilde v\in W^{1,\varphi}(\mathcal{D}(\mm),\nn)$, and 
\begin{equation*}
    \int_{\mathcal{D}(\mm)} \tilde{\varphi}(x,|D\tilde v|)\d\mathscr{H}^{\tx{m}}\geq \int_{\mm_1} \tilde{\varphi}(x,|D\tilde v|)\d\mathscr{H}^{\tx{m}}\geq c\,\int_{\mm}\varphi(x,|Dv|)\d\mathscr{H}^{\tx{m}}.
\end{equation*}
 We sum up what we have just found in the following
\begin{lemma}[Sobolev extension to the double $\mathcal{D}(\mm)$]\label{lemma:est}
    Let $\mm$ be an oriented, compact $\tx{m}$-dimensional Riemannian manifold, with smooth boundary $\partial \mm\neq \emptyset$, and let $\varphi$ be a Young function fulfilling \eqref{v1}-\eqref{v4}.
   Then, if $v\in W^{1,\varphi}(\mm,\R^N)$, its extension $\tilde v$ to the double given by \eqref{double:ext}
   belongs to $W^{1,\tilde\varphi}(\mathcal{D}(\mm),\mathrm{Im}(v))$, with
\begin{equation}
    c\,\|v\|
    _{W^{1,\varphi}(\mm,\R^N)}\leq \|\tilde v\|_{W^{1,\tilde\varphi}(\mathcal{D}(\mm),\R^N)}\leq C\|v\|
    _{W^{1,\varphi}(\mm,\R^N)},
\end{equation}
for some constants $c,C=c,C(\mm,N,\beta,\gamma)>0$, where $\tilde\varphi$ is given by \eqref{double:young} and satisfies the structural assumptions \eqref{v1}-\eqref{v4}. 
\end{lemma}

\begin{remark}\label{remark:bdry}\rm{
   In view of Lemma \ref{lemma:est} and \eqref{same:inf}, it suffices to prove Theorems \ref{theorem1}-\ref{theorem2} in the case $\partial \mm=\emptyset$. Indeed, assume that both  theorems hold for $\varphi$ and  when the reference manifold $\mm$ has no boundary.

  Then, if $\partial \mm\neq \emptyset$, let   $v\in W^{1,\varphi}(\mm,\nn)$ and  $\tilde v\in W^{1,\tilde\varphi}(\mathcal{D}(\mm),\nn)$ be its extension of Lemma \ref{lemma:est}. Since $\tilde{\varphi}$ still satisfies the assumptions of Theorems \ref{theorem1} (by \eqref{same:inf}) and \ref{theorem2}, $\partial \mathcal{D}(\mm)=\emptyset$, $\mathcal{D}(\mm)$ is a $\tx{m}$-dimensional oriented compact Riemannian manifold, then by applying Theorems \ref{theorem1}-\ref{theorem2} in the boundary-less case, we may find a sequence $v_\ell \in C^{\infty}(\mathcal{D}(\mm), \nn)$ such that $v_\ell\to \tilde v$ in $W^{1,\tilde{\varphi}}(\mathcal{D}(\mm),\nn)$. Therefore, $v_\ell|_{\mm}\xrightarrow{\ell\to\infty} v $ in $W^{1,\varphi}(\mm,\nn)$ owing to the identification $\mm=\mm_1\subset \mathcal{D}(\mm)$.}
\end{remark}

\subsection{Vanishing web oscillations}
Here we recall the notion of vanishing web oscillations, which will be instrumental in the proof of Theorem \ref{theorem1}.
We recall that a set $F\subset \mm$ is a \textit{web} if it is compact, of Lebesgue measure zero, and its complement
$\displaystyle \mm \setminus F=\cup_{i=1}^{\bar{\tx{m}}}U_i$ consists of  a finite number of components $U_i$ that are disjoint open connected sets.

\begin{definition}[Vanishing web oscillations \cite{HIMO}]
Let \(\mm\) be an oriented, compact Riemannian manifold, $\partial\mm=\emptyset$, and let  \(w\in W^{1,\varphi}(\mm,\mathbb{R}^N)\).
We say that \(w\) has vanishing web oscillations if for every \(\varepsilon>0\) there exists a web \(F\subset \mm\) such that:
\begin{itemize}
  \item \(\mathrm{fine\text{-}diam}(F):=\max_{i=1,\dots,\bar{\tx{m}}}\operatorname{diam}U_i \le \varepsilon;\)
  \item there exists \(\eta\in W^{1,\varphi}(\mm,\mathbb{R}^N)\cap C^0(\mm,\mathbb{R}^N)\) with
        \(w-\eta\in W^{1,\varphi}_0(U_i,\mathbb{R}^N)\) for every  \(U_i, \ i=1,\dots, \bar{\tx{m}}\). Equivalently \(w\in\eta+W^{1,\varphi}_0(\mm\setminus F,\mathbb{R}^N)\);
  \item for every \(U_i, \ i=1,\dots,\bar{\tx{m}}\), the boundary oscillation
  \[
    \operatorname{osc}(w,\partial U_i)=\max\big\{|\eta(x)-\eta(y)|: x,y\in\partial U_i\big\}
  \]
  satisfies \(\operatorname{osc}(w,\partial U_i)<\varepsilon\).
\end{itemize}
\end{definition}
\noindent
\subsection{A few useful results}
In this section, we collect some auxiliary results that will be used throughout the paper. We begin with a preliminary remark to clarify the relation between our assumptions and the general framework considered in \cite{BGS}.
\begin{remark} \label{rem1}
\normalfont
We point out that assumptions \eqref{v1}-\eqref{v3.1} and \eqref{v4} imply respectively to assumptions 2.1 and 2.2 in \cite{BGS}. In particular, \eqref{v3} and \eqref{v3.1} ensure that (A3) and (A5) are satisfied and \eqref{d2} implies (A4). Thus, our setting fits into the general framework of \cite{BGS}.
\end{remark}
\noindent
For what follows, we recall that the definition of $R_\mm$ is given in Section \ref{notation}, while $\Psi$ and  $\Psi_{\tx{m}-1}$ and their relevant properties are introduced in Section \ref{2.2}. The next result, derived from \cite[Theorem 4.1]{CC}, will serve as a key tool in the proof of Theorem~\ref{theorem1}. 

\begin{theorem} \label{t1}
Let $\mm$ be a Riemannian manifold as in \eqref{M1}, with $\partial \mm=\emptyset$. Let $\vp(x,\cdot)$ be a Young function satisfying \eqref{v1}-\eqref{v3.1}. Assume that either $\tx{m} = 2$, or $\tx{m} \geq 3$ and
\eqn{ip}
$$
\int^\infty \left( \frac{r}{\vp^-_\mm(r)} \right)^{\frac{1}{\tx{m}-2}} \, \d r < \infty
$$
holds, with $\vp_\mm^-$ given by \eqref{def:inf}. Then, there exists a constant $\tilde c \equiv \tilde c(\mm, N)$ such that, if $\tilde{x} \in \mm$ and $r \in (0, R_{\mm})$, it holds
\eqn{4.3}
$$
\text{osc}_{\mathbb{S}^{\tx{m}-1}_r(\tilde{x})} \psi \leq \tilde c \,r \, \Psi_{\tx{m}-1}^{-1} \left( \mint_{\mathbb{S}_r^{\tx{m}-1}(\tilde{x})} \vp(y,|\nabla_\mathbb{S} \psi|) \, \d\mathscr{H}^{\tx{m}-1} \right),
$$
for every sphere $\mathbb{S}_r^{\tx{m}-1}(\tilde{x}) \subseteq \mm$ and any  weakly differentiable function $\psi : \mathbb{S}^{\tx{m}-1}_r(\tilde{x}) \to \mathbb{R}^N$ such that the right hand side of \eqref{4.3} is finite. 

\end{theorem}
\begin{proof}
Owing to \eqref{ip} and \eqref{equiv}, the Young function $\Psi=(\vp_\mm^-)^{**}$ satisfies the assumptions of Theorem 4.1  in \cite{CC}, therefore
$$
\text{osc}_{\mathbb{S}^{\tx{m}-1}_r(\tilde{x})} \psi \leq \tilde c r \, \Psi_{\tx{m}-1}^{-1} \left( \mint_{\mathbb{S}_r^{\tx{m}-1}(\tilde{x})} \Psi(|\nabla_\mathbb{S} \psi|) \, \d\mathscr{H}^{\tx{m}-1} \right).
$$
Since $\Psi(t)\leq \vp_\mm^-(t)\leq \vp(x,t)$ for every $(x,t)\in \mm\times [0,\infty)$ by \eqref{def:inf} and \eqref{equiv}, and since $\Psi_{\tx{m}-1}^{-1}$ is increasing, we can estimate the right hand side of the inequality above as
\begin{align*}
    \Psi_{\tx{m}-1}^{-1} \left( \mint_{\mathbb{S}_r^{\tx{m}-1}(\tilde{x})} \Psi(|\nabla_\mathbb{S} \psi|) \, \d\mathscr{H}^{\tx{m}-1} \right) & \leq  \Psi_{\tx{m}-1}^{-1} \left( \mint_{\mathbb{S}_r^{\tx{m}-1}(\tilde{x})} \varphi(y,|\nabla_\mathbb{S} \psi|) \, \d\mathscr{H}^{\tx{m}-1} \right)
\end{align*}
and the proof is complete.
\end{proof}
\noindent
We now present the counterpart of \cite[Lemma 5.1]{CC} in the framework of Musielak-Orlicz spaces.
\begin{lemma} \label{l0}
Let $\mm$ be a Riemannian manifold as in \eqref{M1}, with $\tx{m} \geq 2$ and $\partial \mm=\emptyset$. Let $\vp(x,\cdot)$ be a Young function satisfying the $\Delta_2$-condition near infinity for every $x\in\mm$. Assume that either \eqref{1a} or \eqref{1b} hold. Given any $\tilde{x} \in \mm$, $R \in (0, R_{\mm})$, any measurable function $w : \mathbb{B}_R^{\tx{m}}(\tilde{x}) \to \mathbb{R}^N$ such that
\[
\int_{\mathbb{B}_R^\tx{m}(\tilde{x})} \vp(y,|w|) \, \d\mathscr{H}^{\tx{m}} < \infty,
\]
and any $\varepsilon > 0$, the set
\[
\left\{ r \in (0, R) : r \Psi_{\tx{m}-1}^{-1} \left( \mint_{\mathbb{S}^{\tx{m}-1}_r(\tilde{x})} \vp(y,|w|) \, \d\mathscr{H}^{\tx{m}-1} \right) \leq \varepsilon \right\}
\]
has positive one-dimensional Lebesgue measure.
\end{lemma}
\begin{proof}
We start by observing that, since $R \in (0,R_\mm)$, using geodesic spherical coordinates centered at $\tilde{x}$, we can assume that $\mathbb{B}^{\tx{m}}_R(\tilde{x})$ agrees with $\tx{B}^{\tx{m}}_R$ in $\mathbb{R}^{\tx{m}}$, and $\mathbb{S}^{\tx{m}-1}_r(\tilde{x})$ is equal to $\tx{S}^{\tx{m}-1}_r$. Arguing by contradiction, let us suppose that there exists some $\tilde{x} \in \mm$ such that \eqref{1a} and \eqref{1b} holds and such that
$$ r\,\Psi^{-1}_{\tx{m}-1} \left (\mint_{\tx{S}^{\tx{m}-1}_r} \vp(y,|w|) \d\mathscr{H}^{\tx{m}-1} \right )>\varepsilon,$$
for almost every $r \in (0,R)$. Therefore,
$$\omega_{\tx{m}-1} r^{\tx{m}-1}\Psi_{\tx{m}-1}\left ( \frac{\varepsilon}{r}\right ) < \int_{\tx{S}^{\tx{m}-1}_r} \varphi(y,|w|) \d\mathscr{H}^{\tx{m}-1},$$
where $\omega_{\tx{m}-1} = \mathscr{H}^{\tx{m}-1}(\tx{S}_1^{\tx{m}-1})$. Therefore, integrating in $r\in (0,R)$, we obtain
\eqn{lhs}
$$\omega_{\tx{m}-1}  \int_0^R r^{\tx{m}-1} \Psi_{\tx{m}-1}\left ( \frac{\varepsilon}{r}\right ) \d r \leq \int_0^R \int_{\tx{S}_r^{\tx{m}-1}} \vp(y,|w|) \d\mathscr{H}^{\tx{m}-1} \d r = \int_{\tx{B}^\tx{m}_R} \vp(y,|w|) \dy <  \infty.$$
Now, when $\tx{m}=2$ we have $\Psi_{\tx{m}-1}=\Psi$ by definition \eqref{def:psim}, and the integral on the leftmost side of \eqref{lhs} diverges by  \eqref{1aa}, that is a contradiction. When $\tx{m} \geq 3$, then
$$ \int_0^R r^{\tx{m}-1} \Psi_{\tx{m}-1}\left ( \frac{\varepsilon}{r} \right ) \d r = \varepsilon \int_{1/R}^{\infty} \frac{\Psi_{\tx{m}-1}(t)}{t^{1+\tx{m}}} \d t.$$
 Applying \cite[Lemma 2.3]{Cib} to $A(\cdot):=\Psi(\cdot)$ (see also \cite[Equations (5.3)-(5.4)]{CC}) and deduce that
$$ \int^{\infty}\frac{\Psi_{\tx{m}-1}(t)}{t^{1+\tx{m}}} \d t = \infty$$
if and only if
$$ \int^{\infty} \left ( \frac{t}{{\Psi}_{\tx{m}-1}^*(t)}\right )^{\tx{m}-1} \d t =\infty.$$
But, by the definition of $\Psi_{\tx{m}-1}$ in \eqref{def:psim}, the last condition above is equivalent to
$$ \int^\infty t^{-\frac{{\tx{m}-1}}{{\tx{m}-2}}} \left ( \int_t^\infty \frac{{\Psi}^*(r)}{r^{1+\frac{\tx{m}-1}{\tx{m}-2}}} \d r\right )^{1-\tx{m}} \d t = \infty;$$
then, by \cite[Lemma 3.3(i)]{CC}, applied to $A(\cdot)=\Psi(\cdot)$ and with the choice $p=\tx{m}-1$, we have that the above condition is equivalent to \eqref{1bb}, therefore we conclude that the left hand side in \eqref{lhs} diverges. This leads to a contradiction also for $\tx{m}\geq 3$, and ends the proof.
\end{proof}
\noindent
We now prove that, under assumptions \eqref{1a} and \eqref{1b}, a function $ w \in W^{1,\varphi}$ has vanishing web oscillation. This generalizes the case treated in \cite[Lemma 5.1]{CC}, where the Young function $A(\cdot)$ is autonomous, i.e., it depends only on the variable $t$. Before proceeding, we make the following observation regarding our assumptions.
\begin{remark}
\normalfont
In \cite[Lemma~5.1]{CC}, the function \( w \) is not assumed to belong to \( L^{\infty}(\mm, \mathbb{R}^N) \). In our setting, however, this assumption is required, since otherwise condition  \eqref{v4} should be stated with \( t  \in [0,\tx{b}\rr^{-\tx{m}/\beta}] \), which for double phase functionals \eqref{dph} corresponds to the bound $q\leq p+\alpha p/\tx{m} $. This does not represent a loss of generality, because we will apply the next lemma to functions \( w \in W^{1,\varphi}(\mm, \nn) \), which are bounded thanks to the compactness of \( \nn \).
\end{remark}

\begin{lemma} \label{l1}
    Let $\mm$ be a Riemannian manifold as in \eqref{M1}, with $\tx{m} \geq 2$ and $\partial \mm=\emptyset$. Let $\vp: \mm \times [0,\infty) \to [0,\infty)$ be a function satisfying \eqref{v1}-\eqref{v4}. Assume either \eqref{1a} or \eqref{1b} holds. Let $w \in (W^{1,\varphi} \cap L^{\infty})(\mm,\R^N)$. Then, for every $\varepsilon>0$, there exists a finite family of geodesic spheres $\mathbb{S}^{\tx{m}-1}_{r_i}(x_i)$, with radii $r_i < \varepsilon$, $i=1,\dots,i_\varepsilon$, and a continuous function $\eta:\mm \to \R^N$ such that
    \eqn{l11}
    $$ w-\eta \in (W^{1,\varphi}_0 \cap L^{\infty})(U_h,\R^N)$$
    and
    \eqn{l12}
    $$ \operatorname*{osc}_{\partial U_h} \eta < \varepsilon,$$
    for $h=1,\dots,h_\varepsilon$, where $\{U_h\}_{h=1,\dots,h_\varepsilon}$ denote the connected components of the web $ \displaystyle \mm \setminus \left ( \cup_{i=1}^{i_\varepsilon} \mathbb{S}^{\tx{m}-1}_{r_i}(x_i) \right )$.
\end{lemma}
\begin{proof}
Let $w \in (W^{1,\varphi} \cap L^{\infty})(\mm, \mathbb{R}^N)$. Let us take an atlas $\{W_j, \phi_j\}_{j=1}^{ \bar{\tx{m}}}$ of $\mm$ and  subsets $V_j \Subset W_j$ such that the family $\{V_j\}_{j=1}^{ \bar{\tx{m}}}$  covers $\mm$, and $\phi_j(V_j) = \tx{B}_1$, $\phi_j(W_j) = \tx{B}_2$, for all $j = 1, \dots,  \bar{\tx{m}}$, with $\tx{B}_1, \tx{B}_2 \subset \mathbb{R}^{\tx{m}}$. Let $\{\chi_j\}_{j=1}^{ \bar{\tx{m}}}$ be a partition of unity subordinate to  $\{V_j\}_{j=1}^{ \bar{\tx{m}}}$. Thanks to Remark \ref{rem1}, we can apply the convolution argument of \cite[Theorem 2.3, Lemma 7.1]{BGS} (adapted to the coordinate domain $V_j$ which we identify with $\tx{B}_1$), and obtain sequences $\{w^j_\ell\}_{\ell \in \mathbb{N}}$ for all $j=1,\dots, \bar{\tx{m}}$, such that
    \eqn{sl}
   $$ \{w^j_\ell\}_\ell \subset C^{\infty}(V_j, \R^N), \quad 
 w^j_{\ell} \xrightarrow{\ell\to\infty} w \quad \text{strongly in } W^{1,\vp}(V_j,\R^N).$$ 
Moreover, since $w\in L^\infty(\mm,\R^N)$ and $w^j_\ell$ are constructed via convolution, we have that
\begin{equation}\label{alm:inf}
    w^j_\ell \xrightarrow{\ell\to\infty} w\quad\text{in $L^q(V_j,\,\mathbb R^N)$}
\end{equation}
 for all $q\in [1,\infty)$ and $j=1,\dots,\bar{\tx{m}}$. Now, we define the sequence
\eqn{successione}
$$w_\ell(x) := \sum_{j=1}^{ \bar{\tx{m}}} \chi_j \cdot w^j_\ell(x)$$
Then, by \eqref{sl}, \eqref{alm:inf}, \eqref{du1}-\eqref{control:vp}, it is immediate to verify that
\eqn{sl1}
$$\{w_{\ell}\}_{\ell} \subset C^\infty(\mm, \mathbb{R}^N), \quad w_\ell \xrightarrow{\ell\to\infty} w\quad \text{strongly in } W^{1,\varphi}(\mm, \mathbb{R}^N)\cap L^q(\mm,\mathbb{R}^N),$$
for every $q\in [1,\infty)$.\\

\noindent Now, by \eqref{d2} and Lemma \ref{l0}, for every $\varepsilon >0$,  it holds
\eqn{fl0}
$$ \tilde{c}\, r \Psi_{\tx{m}-1}^{-1} \left( \mint_{\mathbb{S}^{\tx{m}-1}_r(\tilde{x})} \vp(y,|Dw|) \, \d\mathscr{H}^{\tx{m}-1} \right)  \leq \frac{\varepsilon}{4},$$
for every $\tilde{x} \in \mm$, $\mathbb{S}^{\tx{m}-1}_r(\tilde{x}) \subseteq \mm$ and $r$ in a subset of $(0,R_{\mm})$ with positive measure, where $\tilde{c}$ is the constant appearing in \eqref{4.3}. Now, fix any $R \in (0,\min\{\varepsilon,R_\mm\})$, so that by \eqref{sl1}, for every $\tilde{x} \in \mm$ and every $k>0$, by \eqref{conv2}, we have
$$ \lim_{\ell\to \infty} \int_{\mathbb{B}^{\tx{m}}_R(\tilde{x})}|w_{\ell} - w|\d \mathscr{H}^{\tx{m}}+\int_{\mathbb{B}^{\tx{m}}_R(\tilde{x})}\vp(y,k|Dw_{\ell}-Dw|)\d \mathscr{H}^{\tx{m}} =0.$$
Applying Fubini's theorem, for every $\tilde{x} \in \mm$ and almost every $r \in (0,R)$, we get
\eqn{con1}
$$\lim_{\ell\to \infty} \int_{\mathbb{S}^{\tx{m}-1}_r(\tilde{x})}|w_{\ell} - w|\d \mathscr{H}^{\tx{m}-1}+\int_{\mathbb{S}^{\tx{m}-1}_r(\tilde{x})}\vp(y,2|Dw_{\ell}-Dw|)\d \mathscr{H}^{\tx{m}-1} =0. $$
In particular, for every $\tilde{x} \in \mm$, there exists $r_{\tilde{x}} \in (0,R)$ such that \eqref{fl0} and \eqref{con1} hold with $r=r_{\tilde{x}}$. Therefore, by \eqref{con1}, we have
\eqn{5.12}
$$ \lim_{\ell\to \infty} \int_{\mathbb{S}^{\tx{m}-1}_{r_{\tilde{x}}}(\tilde{x})}\vp(y,|Dw_{\ell}|)\d \mathscr{H}^{\tx{m}-1}=\int_{\mathbb{S}^{\tx{m}-1}_{r_{\tilde{x}}}(\tilde{x})}\vp(y,|Dw|)\d \mathscr{H}^{\tx{m}-1}.$$
Moreover, thanks to \eqref{sl1}, we can extract a subsequence, still labeled as $\{w_\ell\}_\ell$, such that
\eqn{us}
$$ \nr{w_{\ell} -w_{{\ell}-1}}_{W^{1,\vp}(\mm,\R^N)}\leq 2^{-\ell},\quad \text{for every $\ell$.}$$
The family $\{ {\mathbb{B}}^{\tx{m}}_{r_{\tilde{x}}}(\tilde{x})\}_{\tilde{x} \in \mm}$ is an open covering of $\mm$. Hence, by the compactness of $\mm$, for every $\varepsilon>0$, there exists a finite family $\{{\tilde{x}}_i \}_i$, $i=1,\dots,i_\varepsilon$, such that $\{{\mathbb{B}}^{\tx{m}}_i \}_i$, $i=1,\dots, i_\varepsilon$ is a finite covering of $\mm$, where ${\mathbb{B}}^{\tx{m}}_i={\mathbb{B}}^\tx{m}_{r_i}({\tilde{x}}_i)$ and $r_i = r_{{\tilde{x}}_i} < \varepsilon$. 
Clearly, we can suppose that ${\mathbb{B}}^\tx{m}_i \nsubseteq {\mathbb{B}}^{\tx{m}}_j$ if $i\neq j$ and we set $\mathbb{S}^{\tx{m}-1}_i= \partial {\mathbb{B}}^\tx{m}_i$. Hence, by \eqref{con1} applied with $\tilde{x} = \tilde{x}_i$, $i = 1, \dots, i_\varepsilon$, up to extracting a further subsequence, still labeled as $\{w_\ell\}_\ell$, we have
\begin{align} \label{wh}
  \mint_{\mathbb{S}^{\tx{m}-1}_i}|w_\ell - w|\d \mathscr{H}^{\tx{m}-1}  +   \tilde{c}r_i\Psi^{-1}_{\tx{m}-1}\left (\mint_{\mathbb{S}^{\tx{m}-1}_i}\vp(y,2|Dw_{\ell}-Dw|)\d \mathscr{H}^{\tx{m}-1}  \right ) \leq 2^{-\ell-3} \varepsilon.
\end{align}
By the monotonicity of $\vp$
\eqn{fact}
$$\vp(x,s+t)\leq\vp(x,2s)+\vp(x,2t), \quad \text{for every } x \in \mm \text{ and all } s,t \geq 0,$$
 \eqref{wh}, the fact that $t \mapsto \Psi^{-1}_{\tx{m}-1}(x,t)$ is increasing and that $\Psi^{-1}_{\tx{m}-1}$ is a concave function vanishing at $0$, we get 
\begin{align} \label{qe1}
& \tilde{c}  r_i \Psi^{-1}_{\tx{m}-1} \left ( \mint_{\mathbb{S}^{\tx{m}-1}_i} \vp(y,|Dw_{\ell} -Dw_{{\ell}-1}|) \d \mathscr{H}^{\tx{m}-1}\right ) + \mint_{\mathbb{S}^{\tx{m}-1}_i}|w_\ell - w_{\ell -1 }|\d \mathscr{H}^{\tx{m}-1}\notag  \\ & \qquad \leq \tilde{c} r_i \Psi^{-1}_{\tx{m}-1} \left ( \mint_{\mathbb{S}^{\tx{m}-1}_i} \vp(y,2|Dw_{\ell} -Dw|) \d \mathscr{H}^{\tx{m}-1} + \mint_{\mathbb{S}^{\tx{m}-1}_i} \vp(y,2|Dw -Dw_{{\ell}-1}|) \d \mathscr{H}^{\tx{m}-1}\right ) \notag \\ & \qquad \quad +  \mint_{\mathbb{S}^{\tx{m}-1}_i} |w_\ell - w| \d \mathscr{H}^{\tx{m}-1} +   \mint_{\mathbb{S}^{\tx{m}-1}_i} |w - w_{\ell -1}| \d \mathscr{H}^{\tx{m}-1} \notag \\ & \qquad \leq \tilde{c} r_i \Psi^{-1}_{\tx{m}-1} \left ( \mint_{\mathbb{S}^{\tx{m}-1}_i} \vp(y,2|Dw_{\ell} -Dw|) \d \mathscr{H}^{\tx{m}-1} \right )+ \tilde{c} r_i \Psi^{-1}_{\tx{m}-1}\left (\mint_{\mathbb{S}^{\tx{m}-1}_i} \vp(y,2|Dw -Dw_{{\ell}-1}|) \d \mathscr{H}^{\tx{m}-1}\right ) \notag \\ & \qquad \quad + \mint_{\mathbb{S}^{\tx{m}-1}_i} |w_\ell - w| \d \mathscr{H}^{\tx{m}-1} +   \mint_{\mathbb{S}^{\tx{m}-1}_i} |w - w_{\ell -1}| \d \mathscr{H}^{\tx{m}-1} \leq  2^{-\ell-2}\varepsilon.
\end{align}
Thanks to \eqref{1a}-\eqref{1b}, \cite[Lemma 2.3]{Cib} and \cite[Lemma 3.3(i)]{CC}, assumption \eqref{ip} is fulfilled, so applying \eqref{4.3} to $\psi=w_{\ell}-w_{{\ell}-1}$ on $\mathbb{S}^{\tx{m}-1}_i$, we get
\eqn{b3}
$$\text{osc}_{\mathbb{S}^{\tx{m}-1}_i} (w_{\ell}-w_{{\ell}-1}) \leq \tilde{c} r_i \Psi_{\tx{m}-1}^{-1} \left(\mint_{\mathbb{S}^{\tx{m}-1}_i} \vp(y,|\nabla_\mathbb{S} (w_{\ell}-w_{{\ell}-1})|) \, \d\mathscr{H}^{\tx{m}-1} \right).$$
 Moreover, 
\begin{align} \label{forinf}
    \inf_{\mathbb{S}_i^{\tx{m}-1}}|w_{\ell}-w_{{\ell}-1}| & \leq \mint_{\mathbb{S}_i^{\tx{m}-1}}|w_{\ell}-w_{{\ell}-1}| \d \mathscr{H}^{\tx{m}-1}.
\end{align} Therefore, from \eqref{qe1}-\eqref{forinf} we obtain
\begin{align} \label{ai}
    \sup_{\mathbb{S}_i^{\tx{m}-1}}|w_{\ell}-w_{{\ell}-1}| \leq \inf_{\mathbb{S}_i^{\tx{m}-1}}|w_{\ell}-w_{{\ell}-1}| + \text{osc}_{\mathbb{S}_i^{\tx{m}-1}}(w_{\ell}-w_{{\ell}-1}) \leq  
    2^{{-\ell -2}}\varepsilon.
\end{align}
Now, for $\delta >0$, let $T_\delta: \R^N \to \R^N$ be the smooth truncation operator at the level set $\delta$, i.e.,
\eqn{tr1}
$$ T_\delta(\xi)=\xi\, \beta_\delta(|\xi|), \quad \text{for } \xi \in \R^\tx{m},$$
where  $\beta_\delta:[0,\infty)\to [0,\infty)$ is defined by
\begin{equation}\label{def:beta}
    \beta_\delta(s)=\begin{cases}
        1\quad &\text{if $s\in [0,2\delta]$}
        \\
        \frac{4(s-\delta)\delta}{s^2} \quad &\text{if $s\geq 2\delta $}.
    \end{cases}
\end{equation}
We inductively define the sequence $\{\eta_{\ell} \}_{\ell}$ by 
\begin{equation}\label{eta:ell}
    \eta_1 := w_1,\quad \eta_{\ell} - \eta_{{\ell}-1}:= T_{  2^{-\ell}\varepsilon}(w_{\ell}-w_{{\ell}-1}), \quad \text{for } \ell \geq 2.
\end{equation}
Then $\eta_{\ell} \in C^{1}(\mm,\R^N).$ Moreover, for every $\ell\geq 2$, by definition of $\beta(\cdot)$ and $\eta_\ell$, we have
\eqn{cc}
$$ \sup_{\mm}|\eta_{\ell} - \eta_{{\ell}-1}|\leq   
2^{2-\ell}\varepsilon.$$
Again, by definition of $T_\delta$ and \eqref{ai}, we have $T_{  2^{-\ell}\varepsilon}(w_{\ell}-w_{{\ell}-1})=(w_{\ell} -w_{{\ell}-1})$ on each $\mathbb{S}^{\tx{m}-1}_i$, $i=1,\dots,i_\varepsilon$, and for every $\ell \in \mathbb{N}$ 
\eqn{5.22}
$$ \eta_{\ell}=w_{\ell} \quad \text{on} \quad \cup_{i=1}^{i_\varepsilon}\mathbb{S}^{\tx{m}-1}_i.$$
Then, for any weakly differentiable function $v$,  standard properties of the truncation operator \cite[Eq. 5.17]{HIMO} entail
\begin{equation}\label{Dtrunc}
    |DT_\delta (v)|\leq |Dv|\quad\text{almost everywhere.}
\end{equation}
This information together with \eqref{eta:ell} and \eqref{us} yield
\eqn{ab}
$$ \nr{D(\eta_{\ell}-\eta_{{\ell}-1})}_{L^{\vp}(\mm,\R^N)} \leq  \nr{D(w_{\ell}-w_{{\ell}-1})}_{L^{\vp}(\mm,\R^N)} \leq 2^{-\ell}.$$
Then, by \eqref{cc} and \eqref{ab}, $\{ \eta_{\ell} \}_{\ell}$ is a Cauchy sequence in $C^0(\mm,\R^N)$ and $W^{1,\vp}(\mm,\R^N)$, so we can set $$\eta  :=\lim_{\ell\to\infty} \eta_{\ell};$$ 
then, since by \eqref{5.22} it holds $w_{\ell} - \eta_{\ell} \in W^{1,\vp}_0(U_h,\R^N)$, we have that $w-\eta \in W^{1,\vp}_0(U_h,\R^N)$, for $h=1,\dots,h_\varepsilon$. Moreover, since $\eta_{\ell} \to \eta$ uniformly, $\eta_\ell \equiv w_\ell$ on $\mathbb{S}^{\tx{m}-1}_i$ due to \eqref{5.22}, and by \eqref{b3} with $\eta_{\ell}$ in place of $w_{\ell} - w_{{\ell}-1}$, 
\eqref{5.12} and \eqref{fl0} with $\tilde{x}=x_i$, we get 
\begin{align} \label{5.24}
    \text{osc}_{\mathbb{S}^{\tx{m}-1}_i} \eta & = \liminf_{{\ell} \to \infty} \text{osc}_{\mathbb{S}^{\tx{m}-1}_i}\eta_{\ell} \notag \\ & \leq \tilde{c} r_i \liminf_{\ell\to \infty}  \Psi^{-1}_{\tx{m}-1}\left (\mint_{\mathbb{S}^{\tx{m}-1}_i}\vp(y,|\nabla_{\mathbb{S}}\eta_{\ell}  |) \d \mathscr{H}^{\tx{m}-1} \right ) \notag \\ & = \tilde{c} r_i \liminf_{\ell\to \infty}  \Psi^{-1}_{\tx{m}-1}\left (\mint_{\mathbb{S}^{\tx{m}-1}_i}\vp(y,|\nabla_{\mathbb{S}}w_{\ell}|) \d \mathscr{H}^{\tx{m}-1} \right )  \notag \\ &  \leq  \tilde{c}r_i \liminf_{\ell\to \infty}  \Psi^{-1}_{\tx{m}-1}\left (\mint_{\mathbb{S}^{\tx{m}-1}_i}\vp(y,|Dw_{\ell}  |) \d \mathscr{H}^{\tx{m}-1} \right )  \notag  \\ & = \tilde{c}r_i  \Psi^{-1}_{\tx{m}-1}\left (\mint_{\mathbb{S}^{\tx{m}-1}_i}\vp(y,|Dw  |) \d \mathscr{H}^{\tx{m}-1} \right ) \leq  \frac{\varepsilon}{4},
\end{align}
for $i=1,\dots,i_\varepsilon$. The rest of the proof now follows \cite[pp. 576]{CC}: we observe that each connected component $U_h$ of $\mm \setminus \cup_{i=1}^{i_\varepsilon} {\mathbb{S}}_i^{\tx{m}-1}$ lies within some ball ${\mathbb{B}}^{\tx{m}}$ from the finite covering of $\mm$, and its boundary $\partial U_h$ consists of portions of spheres ${\mathbb{S}}_i^{\tx{m}-1}$ intersecting ${\mathbb{B}}^\tx{m}$. Setting ${\mathbb{S}}^{\tx{m}-1} =  \partial {\mathbb{B}}^\tx{m}$, and noting that ${\mathbb{B}}_i^{\tx{m}} \not\subseteq {\mathbb{B}}_j^{\tx{m}}$ when $i \neq j$, each sphere ${\mathbb{S}}_i^{\tx{m}-1}$ intersecting ${\mathbb{B}}^\tx{m}$ must necessarily intersect $\partial {\mathbb{B}}^\tx{m}$ as well. Consequently, for any $z_1, z_2 \in \partial U_h$, we have $z_1 \in {\mathbb{S}}_{i_1}^{\tx{m}-1}$ and $z_2 \in {\mathbb{S}}_{i_2}^{\tx{m}-1}$ for appropriate indices, with both ${\mathbb{S}}_{i_1}^{\tx{m}-1}$ and ${\mathbb{S}}_{i_2}^{\tx{m}-1}$ that have not empty intersection with ${\mathbb{S}}^{\tx{m}-1}$. Selecting points $y_1 \in {\mathbb{S}}_{i_1}^{\tx{m}-1} \cap {\mathbb{S}}^{\tx{m}-1}$ and $y_2 \in {\mathbb{S}}_{i_2}^{\tx{m}-1} \cap {\mathbb{S}}^{\tx{m}-1}$, we deduce from \eqref{5.24} that:
\[
|\eta(z_1) - \eta(z_2)| \leq |\eta(z_1) - \eta(y_1)| + |\eta(y_1) - \eta(y_2)| + |\eta(y_2) - \eta(z_2)| 
\]
\[
\leq \text{osc}_{{\mathbb{S}}_{i_1}^{\tx{m}-1}} \eta + \text{osc}_{{\mathbb{S}}^{\tx{m}-1}} \eta + \text{osc}_{{\mathbb{S}}_{i_2}^{\tx{m}-1}} \eta < \varepsilon,
\]
that is \eqref{l12}, thus completing the proof.
\end{proof}

\section{Density results and absence of Lavrentiev phenomenon} \label{sec5}
\noindent
In this section, we establish the density of smooth maps between manifolds in $W^{1,\vp}$, with consequently absence of Lavrentiev phenomenon. 

\noindent Theorem~\ref{theorem1} relies strongly on Lemma~\ref{l1} and follows the approach of \cite[Theorem 5.1]{HIMO}, adapted to our Musielak-Orlicz setting. Theorem~\ref{theorem2} shows that, in the absence of assumptions \eqref{1a} and \eqref{1b}, the topological condition of $\tx{k}$-connectedness of the target manifold \( \nn \) allows us to recover the density of smooth maps in \( W^{1,\varphi} \).

\begin{proof}[Proof of Theorem \ref{theorem1}]
Owing to Remark \ref{remark:bdry}, it suffices to prove the theorem only in the case $\partial \mm=\emptyset$.
    
    Let $w \in W^{1,\vp}(\mm,\nn)$. Let $\mathbb{S}^{\tx{m}-1}_{r_i}(x_i)$ be a family of geodesic spheres with $r_i < \varepsilon$, $i=1,\dots,i_\varepsilon$, and $U_h$, $h=1,\dots,h_\varepsilon$, the connected component of $ \displaystyle \mm \setminus \left ( \cup_{i=1}^{i_\varepsilon} \mathbb{S}^{\tx{m}-1}_{r_i}(x_i) \right )$ coming from Lemma \ref{l1}. For every $h=1,\dots,h_\varepsilon$, consider the map
    $$ T_\varepsilon \circ w: U_h \to \R^N,$$
    where $T_\varepsilon: \R^N \to \R^N$ is the truncation operator defined by
    $$ T_\varepsilon y = y_* +(y-y_*)\,\beta_\varepsilon(|y-y_*|),$$
    with fixed point $y_* \in w(\partial U_h) \subset \nn \subset \R^N$, and  
with $\beta_\varepsilon$ given by \eqref{def:beta}. We observe that $\beta(t) \in [0,1]$ and 
\eqn{eps}
    $$ |T_\varepsilon y-y_*|\leq 4\varepsilon \quad \text{for every } y \in \R^N.$$
Then $T_\varepsilon w \in W^{1,\vp}(\mm,\R^N)$, and, by \eqref{Dtrunc},  we have that
    \eqn{tr}
    $$ |D(T_\varepsilon w)| \leq |Dw| \quad \text{almost everywhere in } U_h.$$
    Moreover, by Lemma \ref{l1} there exists $u \in (W^{1,\varphi}_0 \cap L^{\infty})(U_h,\R^N)$ and a continuous function $\eta \in W^{1,\vp}(\mm,\R^N)$ such that
    \begin{equation}\label{somma}w=\eta + u\quad\text{ on $U_h$.}
    \end{equation}
  We extend $u\equiv 0$ on $\mm\setminus U_h$, and construct a sequence $\{u_j\}_{j} \subset C^{\infty}_{{c}}(U_h,\R^N)$, following the lines of $\{w_{\ell}\}_{\ell}$ in \eqref{successione}, which  approximates $u$ in $W^{1,\vp}(U_h,\mathbb{R}^N)$, i.e.
    \begin{equation}\label{new:conv}
        u_j\xrightarrow{j\to\infty} u\quad\text{in $W^{1,\varphi}_0(U_h,\R^N)\cap L^q(U_h,\mathbb{R}^N)$,   \quad \text{for } $q\geq 1$.}
    \end{equation}
    We also observe that $\eta - T_\varepsilon(\eta+u_j)$ vanishes in a neighborhood of $\partial U_h$ (possibly depending on $j$). Indeed, $\eta(\partial U_h) \subset \tx{B}_{\varepsilon}(y_*) \subset \R^N$ by \eqref{l11}-\eqref{l12}, and since $\eta$ is continuous, the image of a suitably small neighborhood of $\partial U_h$ lies in $\tx{B}_{3/2\varepsilon}(y_*)$. \\
    
   \noindent Also, since $u_j\in C^\infty_c(U_h,\mathbb{R}^N)$, then in a smaller neighborhood $\mathcal{U_{j,h}}$ of $\partial U_h$, we have that the image of $u_j$ lies in $\tx{B}_{1/2\varepsilon}$.
    Therefore, the image of $\mathcal{U_{j,h}}$ through $\eta+u_j$ lies in $\tx{B}_{2\varepsilon}(y_*)$, so
    it  only remains to notice that $T_\varepsilon y=y$ in $\tx{B}_{2\varepsilon}(y_*)$, whence $\eta-T_\varepsilon(\eta+u_j)$ vanishes in $\mathcal{U_{j,h}}$.   Using these facts, together with the continuity of truncation operator $T_\varepsilon: W^{1,\vp}(U_h,\R^N) \to W^{1,\vp}(U_h,\R^N)$ guaranteed by \eqref{tr},  and recalling \eqref{somma}, we conclude that
\begin{equation}\label{in}
    \begin{split}
         w - T_\varepsilon w &=w-T_\varepsilon(\lim_{j \to \infty} (\eta + u_j)) =\eta +u - \lim_{j \to \infty} T_\varepsilon(\eta + u_j) 
         \\
         &= u + \lim_{j \to \infty}(\eta - T_\varepsilon(\eta + u_j)) \in  W^{1,\vp}_0(U_h,\R^N).
    \end{split}
\end{equation}
We now apply the above truncation procedure to $w$ on each $U_h$ and we denote the resulting map by $\tilde{w}_\varepsilon: \mm \to \R^N$. It follows from   \eqref{tr} and \eqref{in}, that 
    $$\tilde{w}_\varepsilon \in (W^{1,\vp}\cap  L^{\infty})(\mm, \R^N),\quad \tilde{w}_\varepsilon=w\,\text{ on $\partial U_h$ for all $h=1,\dots,h_\varepsilon$},$$
    and
    \eqn{bd}
    $$|D\tilde{w}_\varepsilon(x)| \leq |Dw(x)|, \quad \text{for almost every } x \in \mm.$$
    We observe that the image of $\mm$ via $\tilde{w}_\varepsilon$ is no longer in $\nn$, but we have a control on its oscillation: indeed, thanks to \eqref{eps}, for all $x_1, x_2 \in U_h$, it holds
    \eqn{eps2}
    $$ |\tilde{w}_\varepsilon(x_1) - \tilde{w}_\varepsilon (x_2)| \leq 8 \varepsilon.$$
 Now, for every $U_h$, by Poincar\'e inequality and \eqref{bd}, it holds
    $$ \int_{U_h}\frac{|\tilde{w}_\varepsilon -w|}{\text{diam}U_h} \d\mathscr{H}^{\tx{m}} \leq c \int_{U_h}|D\tilde{w}_\varepsilon -Dw| \d\mathscr{H}^{\tx{m}} \leq c \int_{U_h}|Dw| \d\mathscr{H}^{\tx{m}},$$
   with $c\equiv c(\mm)$. Recalling that $r_i<\varepsilon$, for every $i=1,\dots,i_\varepsilon$, we have that $\text{diam}U_h<c\,\varepsilon$; so, summing up on $h=1,\dots,h_\varepsilon$ the previous estimate and  using that $U_h$ are disjoint, we conclude that
    \eqn{ul}
    $$ \int_{\mm} |\tilde{w}_\varepsilon-w|\d\mathscr{H}^{\tx{m}} \leq c\,\varepsilon\, \nr{Dw}_{L^1(\mm)},$$
    with $c$ as above. Now, \eqref{bd} implies that the sequence $\{D\tilde{w}_\varepsilon\}_{\varepsilon}$ is uniformly bounded in $L^\varphi$; by the reflexivity of such space (see for instance \cite[Theorem 3.6.6]{HHbook}),  for $\varepsilon \to 0$ we have that
    \eqn{lt00}
    $$ D\tilde{w}_\varepsilon \to Dw \quad  \text{weakly in }L^{\vp}(\mm,\R^{N \times \tx{m}}).$$
    Using again \eqref{bd} and  the lower semicontinuity of the $W^{1,\vp}$-norm, we get
    \begin{equation*}
        \|D\tilde{w}_\varepsilon\|_{L^\varphi(\mm)}\leq \|Dw\|_{L^\varphi(\mm)}\leq \liminf_{\varepsilon\to 0} \|D\tilde{w}_\varepsilon\|_{L^\varphi(\mm)},
    \end{equation*}
  whence $\lim_{\varepsilon\to 0} \|D\tilde{w}_\varepsilon\|_{L^\varphi(\mm)}=\|Dw\|_{L^\varphi(\mm)}$, which combined with \eqref{ul}, \eqref{lt00} and the uniform convexity of the space $L^\varphi$ (see for instance \cite[Theorem 3.6.6]{HHbook}) yields 
    \eqn{lt0}
    $$\tilde{w}_\varepsilon \xrightarrow{\varepsilon\to 0} w \quad  \text{strongly in }W^{1,\vp}(\mm,\R^{N}).$$
    Now, for every $\varepsilon>0$, we construct, via the same mollification procedure leading to \eqref{successione} (replacing $w$ with $\tilde{w}_\varepsilon$),  a sequence $\{\tilde{w}_\varepsilon^\ell \}_\ell\subset C^{\infty}(\mm,\R^N)$ such that
    \eqn{lt}
    $$ \tilde{w}_\varepsilon^\ell \xrightarrow{\ell\to \infty} \tilde{w}_\varepsilon \text{ strongly in } W^{1,\varphi}(\mm,\R^N).$$
    Then, by the properties of convolution (see \cite[Eq. (5.32)]{HIMO}) there exists $\ell_\varepsilon>0$ such that, for every $U_h$ and $\ell \in (0,\ell_\varepsilon]$, we have
    \eqn{5.32}
    $$ \text{osc}_{U_h}\tilde{w}_\varepsilon^\ell \lesssim_\mm \text{ess osc}_{U_{h}'} \tilde{w}_\varepsilon \leq 24 \varepsilon, $$
    where $U_{h}' :=\{x \in \mm : \dist(x,U_h) < \ell' \}$ and $\ell \lesssim_\mm \ell' \lesssim_\mm \ell$; note that the last inequality follows from \eqref{eps2} and the remarks after Eq. (5.32) in \cite{HIMO}.
    
    By \eqref{lt0} and \eqref{lt} we deduce that, as $\varepsilon \to 0$, it holds
    \eqn{lt1}
    $$ \tilde{w}_\varepsilon^{\ell_\varepsilon} \to w \text{ strongly in } W^{1,\varphi}(\mm,\R^N).$$
    Now, we project $ \tilde{w}_\varepsilon^{\ell_\varepsilon}$ smoothly onto $\nn$. To this end, we consider the closest point projection $\Pi: \nn_{\tilde{h}} \to \nn$ of a suitable tubular neighborhood $\nn_{\tilde{h}}$ of $\nn$ onto $\nn$, which is a $C^{\infty}$-smooth map. We define the approximating sequence as $$w_{k} := \Pi(\tilde{w}_{\varepsilon_{k}}^{\ell_{k}}),$$ where $\varepsilon_{k} \to 0$ and $\ell_{k}\to 0$ are chosen accordingly. \\

\noindent Let us show that the sequence $\{w_{k}\}_{k}$ is well defined. By \eqref{eps}, we have that $\tilde{w}_\varepsilon(U_h)$ is contained in a $4\varepsilon$-neighborhood of $\nn$, say $\nn_{\tilde{h}/4}$, for $\varepsilon$ small enough. Thanks to \eqref{lt}, we have that $\tilde{w}^\ell_\varepsilon\xrightarrow{\ell\to\infty}\tilde{w}_\varepsilon$ almost everywhere on $\mm$. Therefore, for every mesh $U_h$, we can choose a  point $x_0\in U_h$ and $k_0$ large enough such that $|w_{\varepsilon_k}^{\ell_k}(x_0)-\tilde{w}_{\varepsilon_k}(x_0)|<\tilde{h}/4$, for every $k>k_0$. This information together with \eqref{5.32} imply that $w^{\ell_k}_{\varepsilon_k}(U_h)$ lies in $\nn_{\tilde{h}}$ for $k>k_0$ large enough, and repeating this argument for every mesh $U_h$, $h=1,\dots,h_{\varepsilon_k}$, we deduce that $w^{\ell_k}_{\varepsilon_k}(\mm)\subset \nn_{\tilde{h}}$. Hence $w_k$ is well defined, and
it belongs to $C^{\infty}(\mm,\nn)$ thanks to the smoothness of $\Pi$ and $w^\ell_\varepsilon$.

Finally, by \eqref{lt1}, the chain rule, the  Lipschitz continuity of $\Pi$ and $(\ref{du})_2$, we obtain
    \begin{align*}
        \nr{w_{k} -w}_{W^{1,\vp}(\mm,\nn)} & = \nr{w_{k} -\Pi(w)}_{W^{1,\vp}(\mm,\nn)} \\ & = \nr{w_{k} -\Pi(w)}_{L^1(\mm,\nn)} + \nr{Dw_{k} -D\Pi(w)}_{L^{\vp}(\mm,\nn)} \\ & \lesssim_{\gamma,\nn} \nr{\tilde{w}_{\varepsilon_{k}}^{\ell_{k}} -w}_{L^1(\mm,\R^N)} + \nr{D\Pi(\tilde{w}_{\varepsilon_{k}}^{\ell_{k}})\circ D\tilde{w}_{\varepsilon_{k}}^{\ell_{k}} -D\Pi(w)\circ Dw}_{L^{\vp}(\mm,\R^N)} \\ & \lesssim_{\gamma,\nn} \nr{\tilde{w}_{\varepsilon_{k}}^{\ell_{k}} -w}_{L^1(\mm,\R^N)} + \nr{D\Pi(\tilde{w}_{\varepsilon_{k}}^{\ell_{k}})\circ (D\tilde{w}_{\varepsilon_{k}}^{\ell_{k}} -Dw)}_{L^{\vp}(\mm,\R^N)} \\ & \qquad +\nr{(D\Pi(\tilde{w}_{\varepsilon_{k}}^{\ell_{k}})-D\Pi(w))\circ Dw}_{L^{\vp}(\mm,\R^N)} \xrightarrow{k\to\infty} 0
    \end{align*}
 This concludes the proof. 
\end{proof}

\noindent
We now turn to the proof of Theorem \ref{theorem2}, which relies on the method developed in \cite{HJ}, suitably adapted to our framework. For the reader’s convenience, we briefly recall the main aspects of the construction of the approximation scheme, referring to Sections 2-4 of \cite{HJ} for further details on triangulations, skeletons, and retractions.

\begin{proof}[Proof of Theorem \ref{theorem2}]
In view of Remark \ref{remark:bdry}, we just need to prove the theorem only in the case $\partial \mm=\emptyset$. 

We denote by $T^l$ the $l$-dimensional skeleton of the triangulation $T$ of the manifold $\nn$, that is, the union of all $l$-dimensional simplices. Following the construction of \cite[Sections 2, 3, and 4]{HJ}, for $\varepsilon \in [0,1]$, we denote by $U_\varepsilon T^{\tx{k}}$ a neighborhood of the $\tx{k}$-skeleton $T^{\tx{k}}$ and by $O_\varepsilon T^{\tx{k}} := \operatorname{int}(\nn \setminus U_\varepsilon T^{\tx{k}})$. We also obtain sets $Y^{\tx{n}},Y^{\tx{n}-1},\dots,Y^{\tx{k}+1}$, where $Y^l$ is the set of points chosen inside the $l$-dimensional simplexes (for each $l = \tx{n}, \tx{n}-1,\dots, \tx{k}+1$),  a Lipschitz map $\eta_\varepsilon:\R^N \to \nn$ depending on these sets $Y^\tx{n}, \dots, Y^{\tx{k}+1}$ and $\varepsilon \in [0,1]$, such that
\eqn{eta1}
$$ \eta_{\varepsilon}|_{U_\varepsilon T^\tx{k}}= \text{Id}_{U_\varepsilon T^\tx{k}},$$
\eqn{eta2}
$$\text{Lip}(\eta_\varepsilon) \leq c\varepsilon^{-1},$$
for some constant $c$ independent of the choice of $Y^\tx{n},\dots,Y^{\tx{k}+1}$. Then, we set
$$
Q_\varepsilon T^{l-1} := \operatorname{int}\big(\nn \setminus P_{(Y^l,\varepsilon)}(\nn \setminus W^{\tx{n}-l})\big),
$$
where $P_{(Y^l,\varepsilon)}$ is the Lipschitz retraction map that retracts points in $\nn \setminus W^{\tx{n}-l}$ onto a neighborhood of the $(l-1)$-skeleton, and $W^{\tx{n}-l}$ is a set of singularities of dimension $n-l$, see \cite[pp. 1587]{HJ} for the detailed construction. The set $Q_\varepsilon T^{l-1}$ depends on $Y^{\tx{n}}, \dots, Y^l$ and $\varepsilon$, but there exists a constant $c>0$ such that the maximal number $ {\tx{k}}_{\tx{n}}(\varepsilon)$ of sets $Y^{\tx{n}}$ with pairwise disjoint corresponding sets $Q_{2\varepsilon} T^{\tx{n}-1}$ satisfies
\eqn{kk}
$$ {\tx{k}}_{\tx{n}}(\varepsilon) \ge c \, \varepsilon^{-\tx{n}}.$$ Analogously, for fixed  $Y^{\tx{n}}$, the maximal number $ {\tx{k}}_{\tx{n}-1}(\varepsilon)$ of sets $Y^{\tx{n}-1}$ with pairwise disjoint corresponding sets $Q_{2\varepsilon} T^{\tx{n}-2}$ satisfies $ {\tx{k}}_{\tx{n}-1}(\varepsilon) \ge c \, \varepsilon^{-(\tx{n}-1)}$, and similarly $\tx{k}_l(\varepsilon) \ge c \, \varepsilon^{-l}$. \\

\noindent Now, consider $Y^{\tx{n}}_1, \dots, Y^{\tx{n}}_{ {\tx{k}}_{\tx{n}}(\varepsilon)}$ the family of sets $Y^{\tx{n}}$ such that the corresponding sets $Q_{2\varepsilon,1} T^{\tx{n}-1}, \dots, \break Q_{2\varepsilon,  {\tx{k}}_{\tx{n}}(\varepsilon)} T^{\tx{n}-1}$ are pairwise disjoint.\\

\noindent Given $w \in W^{1,\vp}(\mm,\nn)$, we then have
$$ \int_{\cup_{i=1}^{ {\tx{k}}_\tx{n}(\varepsilon)}w^{-1}(Q_{2\varepsilon,i}T^{\tx{n}-1})} \varphi(x,|Dw|) \d\mathscr{H}^{\tx{m}} = \sum_{i=1}^{ {\tx{k}}_\tx{n}(\varepsilon)} \int_{w^{-1}(Q_{2\varepsilon,i}T^{\tx{n}-1})} \varphi(x,|Dw|) \d\mathscr{H}^{\tx{m}} \leq \nr{ \varphi(  \cdot,|Dw|) }_{L^1(\mm)},$$
and
\begin{equation*}
    \mathscr{H}^{\tx{m}}\big(\cup_{i=1}^{ {\tx{k}}_\tx{n}(\varepsilon)}w^{-1}(Q_{2\varepsilon,i}T^{\tx{n}-1}) \big)=\sum_{i=1}^{\tx{k}_\varepsilon}\mathscr{H}^{\tx{m}}\big(w^{-1}(Q_{2\varepsilon,i}T^{\tx{n}-1}) \big)\leq \mathscr{H}^{\tx{m}}(\mm).
\end{equation*}

\noindent Hence, there exists $j \in \{1,\dots, {\tx{k}}_\tx{n}(\varepsilon) \}$ such that
$$ \int_{w^{-1}(Q_{2\varepsilon,j}T^{\tx{n}-1})}\varphi(x,|Dw|) \d\mathscr{H^{\tx{m}}} \leq \frac{1}{ {\tx{k}}_\tx{n}(\varepsilon)} \nr{\varphi(\cdot,|Dw|)}_{L^1(\mm)} \overset{\eqref{kk}}{\leq} c \varepsilon^\tx{n} \nr{\varphi(\cdot,|Dw|)}_{L^1(\mm)},$$
and
\begin{equation*}
    \mathscr{H}^{\tx{m}}\big(w^{-1}(Q_{2\varepsilon,j}T^{\tx{n}-1}) \big)\leq c\,\varepsilon^{\tx{n}}\mathscr{H}^{\tx{m}}(\mm).
\end{equation*}
Fix the set $Y^\tx{n}_j$. Via the same reasoning, we find $Y^{\tx{n}-1}$ such that
$$ \int_{w^{-1}(Q_{2\varepsilon}T^{\tx{n}-2})}\varphi(x,|Dw|) \d\mathscr{H}^{\tx{m}} \leq c \varepsilon^{\tx{n}-1}\nr{\varphi(\cdot,|Dw|)}_{L^1(\mm)},$$
and
\begin{equation*}
    \mathscr{H}^{\tx{m}}\big(w^{-1}(Q_{2\varepsilon}T^{\tx{n}-2}) \big)\leq c\,\varepsilon^{\tx{n}-1}\mathscr{H}^{\tx{m}}(\mm).
\end{equation*}
Proceeding inductively, we find sets $Y^{\tx{n}},Y^{\tx{n}-1},\dots,Y^{ \tx{k}+1}$  such that, for $l=\tx{k},\dots,\tx{n}-1$,
$$ \int_{w^{-1}(Q_{2\varepsilon}T^{l})}\varphi(x,|Dw|) \d\mathscr{H}^{\tx{m}} \leq c \varepsilon^{ l+1}\nr{\varphi(\cdot,|Dw|)}_{L^1(\mm)},$$
and
\begin{equation*}
    \mathscr{H}^{\tx{m}}\big(w^{-1}(Q_{2\varepsilon}T^{l}) \big)\leq c\,\varepsilon^{l+1}\mathscr{H}^{\tx{m}}(\mm).
\end{equation*}
Hence, since
$$ O_{2\varepsilon}T^ \tx{k} = \displaystyle \cup_{i=1}^{\tx{n}- \tx{k}}Q_{2\varepsilon}T^{\tx{n}-i},$$
there exists a constant $c$ such that
\eqn{3}
$$ \int_{w^{-1}(O_{2\varepsilon}T^ \tx{k})} \varphi(x,|Dw|) \d\mathscr{H}^{\tx{m}} \leq  c (\varepsilon^\tx{n} + \dots + \varepsilon^{ \tx{k}+1}) \nr{\varphi(\cdot,|Dw|)}_{L^1(\mm)} \leq c \varepsilon^{ \tx{k}+1} \nr{\varphi(\cdot,|Dw|)}_{L^1(\mm)}, $$
and, similarly,
\eqn{4}
$$ \mathscr{H}^{\tx{m}}\big(w^{-1}(O_{2\varepsilon}T^ \tx{k})\big) \leq c \varepsilon^{\tx{k}+1}\mathscr{H}^\tx{m}(\mm).$$
Let us assume that for every $\varepsilon \in [0,1]$ the sets $Y^\tx{n}_{\varepsilon}, \dots, Y^{ \tx{k}+1}_\varepsilon$ are chosen in such a way that \eqref{3}-\eqref{4} hold, and divide the proof into two cases.
\newline
\textbf{Case 1}: $\gamma \in (1, \tx{k}+1)$. \\
Let $\eta_\varepsilon: \R^N \to \nn$ be the mapping satisfying \eqref{eta1} and \eqref{eta2} (which depends on the choice of $Y^\tx{n}_\varepsilon, \dots, Y^{ \tx{k}+1}_\varepsilon$), and let us prove that  $\eta_\varepsilon (w) \to w$ in $W^{1,\varphi}(\mm,\nn)$. Since by \eqref{eta1}  and \eqref{4},  $\eta_\varepsilon (w)\neq w$ on a set of arbitrary small measure,  and the the maps $\{\eta_\varepsilon(w)\}_\varepsilon$ are uniformly bounded, by dominated convergence theorem we have 
\eqn{im}
$$\int_{\mm}|\eta_\varepsilon (w) - w|^q \d\mathscr{H}^{\tx{m}} \xrightarrow{\varepsilon\to 0} 0,\quad\text{for all $q\in [1,\infty)$.}$$ 
We are left to prove the gradient convergence. 
By \eqref{eta1}, we have $\eta_\varepsilon(w)=w$ on $w^{-1}(U_\varepsilon T^ \tx{k})$, hence $D(\eta_\varepsilon (w)) = Dw$ almost everywhere on $w^{-1}(U_\varepsilon T^ \tx{k})$; then by \eqref{subadd}, \eqref{eta2}, $(\ref{du})$ and \eqref{3} we deduce
\begin{align} \label{im2}
 \int_{\mm} \varphi(x,|D(\eta_\varepsilon (w)) - Dw|)\d\mathscr{H}^{\tx{m}} &= \int_{w^{-1}(O_\varepsilon T^ \tx{k})} \varphi(x,|D(\eta_\varepsilon (w)) - Dw|)\d\mathscr{H}^{\tx{m}}  \noindent  \notag\\ &\leq c \left (\frac{1}{\varepsilon^\gamma} \int_{w^{-1}(O_\varepsilon T^ \tx{k})}\varphi(x,|D w|) \d \mathscr{H}^{\tx{m}} + \int_{w^{-1}(O_\varepsilon T^ \tx{k})} \varphi(x,|D w|)\d \mathscr{H}^{\tx{m}}\right )  \notag \\ & \leq c  \left (\varepsilon^{ \tx{k}+1 -\gamma} \nr{\varphi(\cdot,|Dw|)}_{L^1(\mm)} + \varepsilon^{\tx{k}+1}\nr{\varphi(\cdot,|Dw|)}_{L^1(\mm)} \right ), 
\end{align}
for some positive constant $c \equiv c(\gamma)$. So, letting $\varepsilon \to 0$ and using $\gamma< \tx{k}+1$, we get 
\eqn{c1}
$$\int_{\mm} \varphi(x,|D(\eta_\varepsilon (w)) - Dw|)\d\mathscr{H}^{\tx{m}} \to 0.$$
Now, let \(\{\tilde{w}_{\ell}\}_{\ell} \subset C^\infty(\mm, \mathbb{R}^N)\) be the sequence obtained via convolution and partition of unity as \eqref{successione}, and satisfying \eqref{sl1}. Denote by \(\pi: \mathbb{R}^N \to \mathbb{R}^N\) a smooth extension of the nearest point projection from a suitable tubular neighborhood of \(\nn\) onto \(\nn\). Since \(\tilde{w}_{\ell} \to w\) in measure, and using \eqref{4}, we can select, for every \(\varepsilon > 0\), an index \({\ell}(\varepsilon)\) such that
\eqn{ind}
$$ \mathscr{H}^{\tx{m}}\big((\pi \circ \tilde{w}_{{\ell}(\varepsilon)})^{-1}(\R^N \setminus U_\varepsilon T^ \tx{k})\big) \xrightarrow{\varepsilon \to 0}0.$$
See also the discussion of \cite[end of page 1589]{HJ}.\\

\noindent By the Lipschitz continuity of $\pi$ and \eqref{du1}-\eqref{du}, we also have 
\eqn{ind1}
$$\pi \circ \tilde{w}_{\ell} \to \pi \circ w=w \quad \text{in } W^{1,\varphi}(\mm,\R^N).$$ Hence, by \eqref{ind} and \eqref{ind1}, up to subsequence, we can assume
\eqn{c2}
$$ \int_{\mm} \varphi(x,|D(\pi \circ \tilde{w}_{{\ell}(\varepsilon)}) -Dw|) \d\mathscr{H}^{\tx{m}} < \varepsilon^{ \gamma+1}.$$
Define $v_\varepsilon := \pi \circ \tilde{w}_{{\ell}(\varepsilon)}$, and let us show that the following sequence of Lipschitz function
$$ w_\varepsilon := \eta_\varepsilon \circ v_\varepsilon \in \mathrm{Lip}(\mm,\nn)$$
converges in $W^{1,\varphi}(\mm,\nn)$ to $w$.
Once this is established, an additional regularization step (for instance, via convolution) applied to the sequence $w_\varepsilon$, coupled with the nearest point projection \cite[Lemma 2]{HJ2} yields the desired density result.\\

\noindent By \eqref{eta1} and \eqref{ind}, $w_\varepsilon \neq v_\varepsilon$ on a set of measure convergent to $0$ as $\varepsilon \to 0$, hence  by \eqref{ind1}
\eqn{s1}
$$\int_{\mm}|w_\varepsilon-w|\d\mathscr{H}^\tx{m} \to 0.$$ Now, let $A_\varepsilon=v_\varepsilon^{-1}(U_\varepsilon T^ \tx{k})$, so that
    \begin{align} \label{III}
     & \int_{\mm} \varphi (x,|Dw_\varepsilon - D (\eta_\varepsilon(w))|) \d\mathscr{H}^{\tx{m}} 
      \notag  \\
       & \qquad = \int_{A_\varepsilon} \varphi(x,|Dw_\varepsilon - D (\eta_\varepsilon(w))|) \d\mathscr{H}^{\tx{m}} + \int_{\mm \setminus A_\varepsilon} \varphi(x,|Dw_\varepsilon - D(\eta_\varepsilon( w))|) \d\mathscr{H}^{\tx{m}} 
      \notag  \\
         & \qquad =: \tx{I} + \tx{II}. 
    \end{align}
Note that $Dw_\varepsilon=D(\eta_\varepsilon( v_\varepsilon))=D v_\varepsilon$ almost everywhere on $A_\varepsilon$. Hence by \eqref{subadd}, \eqref{c1} and \eqref{c2}, we get
\eqn{cal}
$$ \tx{I}\leq  c\,\int_{A_\varepsilon}\varphi(x,|Dv_\varepsilon -Dw|)\d\mathscr{H}^{\tx{m}} + c\,\int_{A_\varepsilon}\varphi(x,|Dw -D(\eta_\varepsilon(w))|)\d\mathscr{H}^{\tx{m}} \xrightarrow{\varepsilon\to 0} 0,$$
where $c >0$. For $\tx{II}$, by \eqref{subadd} we observe that 
\begin{align*}
    \tx{II} & = \int_{\mm \setminus A_\varepsilon} \varphi(x,|D\eta_\varepsilon(v_\varepsilon)\circ Dv_\varepsilon - D\eta_\varepsilon(w) \circ Dw|) \d \mathscr{H}^{\tx{m}} \\ & \leq  c\int_{\mm \setminus A_\varepsilon} \varphi(x,|D\eta_\varepsilon(v_\varepsilon) \circ Dv_\varepsilon-D\eta_\varepsilon(v_\varepsilon)\circ Dw|) \d\mathscr{H}^{\tx{m}}  \\ & \qquad + c\int_{\mm \setminus A_\varepsilon} \varphi(x,|D\eta_\varepsilon(v_\varepsilon) \circ Dw-D\eta_\varepsilon(w)\circ Dw|) \d\mathscr{H}^{\tx{m}} \\ & \leq c \varepsilon^{-\gamma} \int_{\mm \setminus A_\varepsilon} \varphi(x,|Dv_\varepsilon - Dw|) \d\mathscr{H}^{\tx{m}} + c\varepsilon^{-\gamma}\int_{\mm \setminus A_\varepsilon} \varphi(x,|Dw|) \d\mathscr{H}^{\tx{m}},
\end{align*}
where in the last inequality we used \eqref{eta2} and $(\ref{du})$, with $c \equiv c(\gamma) >0$.
By \eqref{c2}, the first integral in the right hand side converges to $0$ as $\varepsilon \to 0$, while for the second one we observe that
\begin{align*}
    \varepsilon^{-\gamma}\int_{\mm \setminus A_\varepsilon} \varphi(x,|Dw|) \d\mathscr{H}^{\tx{m}} & \leq \varepsilon^{-\gamma} \int_{w^{-1}(O_{2\varepsilon}T^ \tx{k})}\varphi(x,|Dw|) \d\mathscr{H}^{\tx{m}} +  \varepsilon^{-\gamma} \int_{v_\varepsilon^{-1}(\overline{O_\varepsilon T^ \tx{k}})\setminus w^{-1}(O_{2\varepsilon}T^ \tx{k})}\varphi(x,|Dw|) \d\mathscr{H}^{\tx{m}} \\ & \qquad + \varepsilon^{-\gamma} \int_{v^{-1}_{\varepsilon}(\R^N \setminus \nn)}\varphi(x,|Dw|) \d\mathscr{H}^{\tx{m}} = {\tx{II}_1} + {\tx{II}_2} + {\tx{II}_3}.
\end{align*}
For ${\tx{II}_1}$ we observe that, by \eqref{3},
\eqn{i1}
$${\tx{II}_1} \leq c\varepsilon^{ \tx{k}+1-\gamma} \nr{\varphi(\cdot, |Dw|)}_{L^1(\mm)}.$$
For ${\tx{II}_3}$ we observe that the map $\tilde w_{{\ell}}$ converges to $w$ in measure, and $\mathrm{Im}(w)\subset \nn$, therefore
\begin{equation*}
    \mathscr{H}^{\tx{m}}\big((\pi \circ \tilde{w}_\ell)^{-1}(\R^N \setminus \nn)\big) \xrightarrow{\ell\to\infty}0;
\end{equation*}
thus, choosing a proper subsequence $\tilde w_{\ell(\varepsilon)}$ and consequently a subsequence of $v_\varepsilon=\pi\circ \tilde w_{\ell(\varepsilon)} $, we can assume that \begin{equation} \label{i2}
    \tx{II}_3 < \varepsilon.
\end{equation}
For what concerns $\tx{II}_2$, we use that $\pi \circ \tilde{w}_\ell$ converges to $w$ in measure and the distance between $\overline{O_\varepsilon T^ \tx{k}}$ and $\overline{U_{2\varepsilon} T^ \tx{k}}=\nn\setminus O_{2\varepsilon} T^{\tx{k}} $ is positive, thus getting 
$$ \mathscr{H}^{\tx{m}}\big((\pi \circ \tilde w_{\ell})^{-1}(\overline{O_\varepsilon T^ \tx{k}})\setminus w^{-1}(O_{2\varepsilon}T^ \tx{k})\big) \xrightarrow{\ell\to\infty} 0.$$
So again, up to subsequence of $\{v_\varepsilon \}$, we can assume that
\eqn{i3}
$${\tx{II}_2} < \varepsilon.$$ 
Therefore, by \eqref{i1}-\eqref{i3}, and using that $\gamma <  \tx{k}+1$, we obtain that also $\tx{II} \xrightarrow{\varepsilon\to 0} 0$. Using this information and \eqref{cal} into \eqref{III} we obtain $\int_{\mm} \varphi(x,|Dw_\varepsilon - D (\eta_\varepsilon( w))|) \d\mathscr{H}  \to 0$, which together with \eqref{s1}, \eqref{c1} yields $w_\varepsilon \to w$ in $W^{1,\varphi}(\mm,\nn)$. This concludes the proof in this case.  \newline
\textbf{Case 2}: $\gamma= \tx{k}+1$. \newline
We will show that $C^{\infty}(\mm,\nn)$ is weakly dense in $W^{1,\vp}(\mm,\nn)$. Let us point out what changes respect to the previous case. First of all, we observe that \eqref{im} holds, whereas \eqref{im2} becomes
\begin{align} \label{uu1}
    \int_{\mm} \varphi(x,|D(\eta_\varepsilon(w)) - Dw|) \d \mathscr{H}^{\tx{m}} \leq c  \nr{\varphi(\cdot,|Dw|)}_{L^1(\mm)}  .
\end{align} Similarly, \eqref{s1} holds and inequality \eqref{i1} becames
$ \tx{II}_1 \leq c \nr{\varphi(\cdot,|Dw|)}_{L^1(\mm)}.$ Therefore, 
\eqn{uu2}
$$ \int_{\mm} \varphi(x,|Dw_\varepsilon - D (\eta_\varepsilon(w))|) \d\mathscr{H}^{\tx{m}} \leq \tx{I} + \tx{II}_2 + \tx{II}_3 + c \nr{\varphi(\cdot,|Dw|)}_{L^1(\mm)}. $$
 Now, by \eqref{c2}, \eqref{cal}, \eqref{i2}-\eqref{uu2}  and  \eqref{subadd},  it follows that
\begin{equation*}
    \|Dw_\varepsilon\|_{L^{\varphi}(\mm,\nn)}\leq c\,(1+\|Dw\|_{L^{\varphi}(\mm,\nn)}).
\end{equation*}
Via an additional convolution and nearest point projection argument applied to $w_\varepsilon$ (see \cite[Lemma 2]{HJ2}) we obtain a sequence $\{\tilde{w}_k\}_k\subset {C^\infty(\mm,\nn)}$, $\tilde{w}_k\to w$ in $L^1(\mm,\nn)$, which is uniformly bounded in $W^{1,\varphi}(\mm,\nn)$. The reflexivity of such space finally yields  the desired result, i.e.,
$\tilde w_k \to w$ weakly in $W^{1,\varphi}(\mm,\nn)$.
\end{proof}
\noindent
We conclude this section by observing that Corollary \ref{corollary}, which establishes the absence of the Lavrentiev phenomenon in the settings we have considered, is a straightforward consequence of Theorem \ref{theorem1} and \ref{theorem2}.
\section{Counterexample} \label{sec4}
\noindent
This final section is devoted to proving the fundamental role played by assumption \eqref{v4}. When this condition fails, counterexamples arise, and this is precisely the content of Theorem \ref{tcontr}. We consider the double phase functional \eqref{dph}. In this context condition \eqref{v4} is implied by \eqref{1pq},
where $\alpha$ is the H\"older exponent of the function $a(\cdot)$. Assuming \eqref{pq} and taking the sphere
$\tx{S}^{N-1}_{{\Lambda}}$ as target manifold, we construct a map $u \in u_0 + W^{1,\varphi}((-1,1)^{\tx{m}},\tx{S}^{N-1}_{{\Lambda}})$ 
that cannot be approximated by smooth sphere-valued maps, where $\Lambda$ will be specified below. The construction, adapted from \cite{BDS}, provides the first vectorial counterexample and extends naturally to maps between manifolds. Note that, \eqref{1pq} is sharp when $p \leq \tx{m}$. We remark once again that the counterexample arises with a target manifold that is $(N-2)$-connected and even when $p \geq \tx{m}$, independently of the smoothness of the boundary data or the domain. Moreover, the functional has the Uhlenbeck structure, i.e. it is radial in the $z$-variable. So, take 
\begin{equation}\label{new:double}
    \varphi (x,|z|)= |z|^p +a(x)|z|^q, \quad z \in \R^{N\times \tx{m}},
\end{equation}
 where $a\in C^{\alpha}([-1,1]^\tx{m})$ will be defined below,  and assume \eqref{pq}. We set $Q_\mm:=(-1,1)^\tx{m}$
and $\overline{Q}_{\mm}:= [-1,1]^\tx{m}$. We start considering a Cantor set constructed in the following way: for \(\lambda \in (0,1/2)\) we take $\mathcal{C}_{\lambda,0}:=[-1/2,1/2]$,  then
 we define \(\mathcal{C}_{\lambda,k+1}\) inductively by removing the open middle portion of length \(1-2\lambda\) from each interval in \(\mathcal{C}_{\lambda,k}\), and we set $\mathcal{C}_\lambda := \bigcap_{k\ge 1} \mathcal{C}_{\lambda,k}$. The corresponding Cantor measure $\mu_\lambda$ is defined as the weak limit of the measures $\mu_{\lambda,k}:=(2\lambda)^{-k}\mathbb{I}_{\mathcal{C}_{\lambda,k}} \dx$, where $(2\lambda)^{-k}$ is chosen such that $\mu_{\lambda,k}([-1/2,1/2])=1=\mu([-1/2,1/2])$. Then, $\mu_\lambda(\R)=1$ and $\supp \mu_\lambda = \mathcal{C}_\lambda$. Finally, the $\tx{m}$-dimensional Cantor set $\mathcal{C}^{\tx{m}}_\lambda$ and its distribution $\mu_\lambda^{\tx{m}}$ are the Cartesian product of $\mathcal{C}_\lambda$ and $\mu_\lambda$, so that $\mathcal{C}^{\tx{m}}_\lambda = \cap_{k \geq 1}\mathcal{C}^{\tx{m}}_{\lambda,k}$.\\ 
 
 \noindent Now, by \eqref{pq}, we can  choose $p_0 >p$ such that 
 \eqn{qp0}
 $$q>p_0+\alpha \max \left \{1,\frac{p-1}{\tx{m}-1} \right \}.$$ Let us split the analysis in three cases according to the relation of $p_0$ with the dimension $\tx{m}$.  We start considering the case $p_0 \in (1,\tx{m})$. Let $\mathcal{C}:= \mathcal{C}^{\tx{m}-1}_\lambda \times \{0\}$, so we have
 $\dim(\mathcal{C}):=\frac{(\tx{m}-1)\log 2}{\log(1/\lambda)}$. First of all, we choose $\lambda \in (0,1/2)$ such that 
 \eqn{p0}
 $$p_0 = \tx{m} - \dim(\mathcal{C}).$$
 Set $x:=(\bar{x},x_{\tx{m}}) \in \R^{\tx{m}-1} \times \R$, and from \cite[Lemma 5]{BDS} we find maps $\chi_*, \chi_a \in C^{\infty}(\R^\tx{m} \setminus \mathcal{C})$, such that
\begin{equation}\label{test1}
    \mathbb{I}_{\{\dist(\bar{x},\mathcal{C}^{\tx{m}-1}_\lambda)\leq 2|x_{\tx{m}}| \}} \leq \chi_* \leq \mathbb{I}_{\{\dist(\bar{x},\mathcal{C}^{\tx{m}-1}_\lambda)\leq 4|x_{\tx{m}}| \}}, \quad \mathbb{I}_{\{\dist(\bar{x},\mathcal{C}^{\tx{m}-1}_\lambda)\leq |x_{\tx{m}}|/2 \}} \leq \chi_a \leq \mathbb{I}_{\{\dist(\bar{x},\mathcal{C}^{\tx{m}-1}_\lambda)\leq 2|x_{\tx{m}}| \}},
\end{equation}
 \begin{equation}\label{test2}
     |D \chi_*| \lesssim_{\tx{m}} |x_{\tx{m}}|^{-1} \mathbb{I}_{\{2|x_{\tx{m}}| \leq \dist(\bar{x},\mathcal{C}^{\tx{m}-1}_\lambda)\leq 4|x_{\tx{m}}| \}}, \quad |D \chi_a| \lesssim_{\tx{m}} |x_{\tx{m}}|^{-1} \mathbb{I}_{\{|x_{\tx{m}}|/2 \leq \dist(\bar{x},\mathcal{C}^{\tx{m}-1}_\lambda)\leq 2|x_{\tx{m}}| \}}.
 \end{equation}
Here $\mathbb{I}_A$ denotes the characteristic function of a set $A$.\\

\noindent Now, pick $\theta \in C^{\infty}_c(0,\infty)$ satisfying $\mathbb{I}_{(1/2,\infty)} \leq \theta \leq \mathbb{I}_{(1/4,\infty)} $, $\nr{\theta'}_\infty\leq 6$, and consider
\eqn{ztilde}
$$ \tx{Z}_{\tx{m}}(x) := \frac{|\bar{x}|^{1-\tx{m}}}{\mathscr{H}^{\tx{m}-1}(\partial \tx{B}_1)} \theta \left ( \frac{|\bar{x}|}{|x_{\tx{m}}|}\right ) \left [\begin{matrix}
0 & -\bar{x}^t \\
\bar{x} & 0 
\end{matrix} \right ], \quad \tx{Z}:=\left ( \mu_\lambda^{\tx{m}-1} \times \delta_0\right )*\tx{Z}_{\tx{m}}, \quad 
\tilde{\tx{Z}} := \sum_{i=1}^N \tx{Z} \otimes e_i, \quad
\tx{b} := 
\text{div}(\tilde{\tx{Z}}).$$
We point out that $\delta_0$ is the delta measure centered in zero and the symbol "$t$" denotes the transposition. Using \cite[Proposition 2 and 14]{BDS}, we have $\tx{Z}_{\tx{m}} \in W^{1,1}_{\loc}(\R^\tx{m}, \R^\tx{m} \otimes \R^\tx{m}) \cap C^{\infty}(\R^\tx{m}\setminus \{0\},\R^\tx{m} \otimes \R^\tx{m})$, $\tx{Z} \in W^{1,1}(Q_\mm, \R^\tx{m} \otimes \R^\tx{m}) \cap C^{\infty}(\overline{Q}_{\mm}\setminus \mathcal{C},\R^\tx{m} \otimes \R^\tx{m})$, $\tilde{\tx{Z}} \in W^{1,1}( Q_\mm, \R^N \otimes \R^\tx{m} \otimes \R^\tx{m}) \cap C^{\infty}(\overline{ Q_\mm}\setminus \mathcal{C},\R^N \otimes \R^\tx{m} \otimes \R^\tx{m})$ and $\tx{b} \in L^1( Q_\mm,\R^{N \times \tx{m}})\cap C^{\infty}(\overline{ Q_\mm}\setminus \mathcal{C}, \R^{N\times \tx{m}})$. 

We remark that the matrix 
$\tx{b}=\{\tx{b}^i_{j}\}^{i=1,\dots,N}_{j=1,\dots,\tx{m}} $ satisfies $\tx{b}^i_{j}=b_j$ for all $i=1,\dots,N$, where $\{b_j\}_{j=1,\dots,\tx{m}}$ is the vector field constructed in \cite[Definition 9]{BDS}. In particular by \cite[Proposition 18]{BDS}, we have
\begin{equation}\label{zero:b}
    \int_{Q_\mm} \langle Dw,\tx{b}\rangle\,\d x=0\quad\text{for all $w\in C^\infty_c(Q_\mm,\mathbb{R}^N)$.}
\end{equation}
Here $\langle A,B\rangle=\sum_{\substack{i=1,\dots,N \\ j=1,\dots,\tx{m}}}A^i_j B_j^i$ denotes the scalar product between two matrices $A,B\in \mathbb{R}^{N\times \tx{m}}$.\\

\noindent We now take $\phi \in C^{1}_c(Q_\mm)$ such that $\mathbb{I}_{(-3/4,3/4)^\tx{m}} \leq \phi \leq \mathbb{I}_{(-5/6,5/6)^\tx{m}}$ and $|D\phi| \lesssim_{\tx{m}} 1$, and define the functions
\eqn{ffs}
$$ u_*(x):=  \frac{1}{2}\text{sgn}(x_{\tx{m}}) \chi_*(x), \quad a(x):=|x_{\tx{m}}|^\alpha \chi_a(x), \quad \tilde{u}(x):= (1-\phi(x))u_*(x).$$
Note that  by contruction of $\mathcal{C}$, \eqref{test1} and $\phi$, we have $u_* \in L^{\infty}(\overline{ Q_\mm},\R) \cap W^{1,1}( Q_\mm,\R)\cap C^{\infty}(\overline{ Q_\mm}\setminus \mathcal{C},\R)$, $\tilde{u} \in C^{\infty}(\overline{ Q_\mm},\R)$, $0 \leq a(\cdot) \in C^{\alpha}(\overline{ Q_\mm})$.\\

\noindent Let us consider 
\begin{equation}\label{def:chiamo}
    u_{**}(x) :=  (u_*(x), 1, \dots,1 ) \in \R^{N}, \quad  u^{\tx{S}}_*(x) :=\frac{u_{**}(x)}{|u_{**}|} \in \tx{S}^{N-1},
\end{equation}
$$ \tilde{\tilde{u}}(x) := (\tilde{u}(x),1,\dots,1) \in \R^N, \quad \tilde{u}^{\tx{S}}(x):= \frac{\tilde{\tilde{u}}(x)}{|\tilde{\tilde{u}}|} \in \tx{S}^{N-1}.$$
Note that $u^{\tx{S}}_* \in W^{1,1}(  Q_\mm,\tx{S}^{N-1})\cap C^{\infty}(\overline{  Q_\mm}\setminus \mathcal{C},\tx{S}^{N-1})$, and $\tilde{u}^{\tx{S}}\in C^{\infty}(\overline{Q}_\mm,\tx{S}^{N-1})$.
Then, by \cite[Lemma 6, Proposition 15, Corollary 16]{BDS} we have
\eqn{cases}
$$
\begin{cases}
   & |Du_*^{\tx{S}}| \lesssim_{\tx{m},N} |x_{\tx{m}}|^{-1}\mathbb{I}_{\{2|x_{\tx{m}}| \leq \dist(\bar{x}, \mathcal{C}^{\tx{m}-1}_\lambda)\leq 4|x_{\tx{m}}| \}}, \ Du_*^{\tx{S}} \in L^{p_0,\infty}(Q_\mm,\R^{N\times \tx{m}}), \\ & \{x \in   Q_\mm: |Du_*^{\tx{S}}| \neq 0 \} \subset \{x \in   Q_\mm: a(x)=0 \} \\ & \tx{b} \in L^{p_0',\infty}(Q_\mm,\R^{N\times \tx{m}}), \quad |\tx{b}| \lesssim_{\tx{m},N} |x_{\tx{m}}|^{1-p_0} \mathbb{I}_{\{\dist(\bar{x},\mathcal{C}_\lambda^{\tx{m}-1}) \leq |x_{\tx{m}}|/2\}}  \\ & \{x \in   Q_\mm: |\tx{b}| > 0\}\subset \{x \in   Q_\mm: a(x)=|x_{\tx{m}}|^{\alpha}\}.
\end{cases}
$$
Then, by $(\ref{cases})_{1,2}$, and since $p_0>p$, it holds $\varphi(x,|Du_*^{\tx{S}}|)= |Du_*^{\tx{S}}|^p\in L^{1}(Q_\mm)$. Observe also that
\eqn{ac}
$$ \varphi(x,|\tx{b}|) \geq |x_{\tx{m}}|^\alpha |\tx{b}|^q, $$
so, using $(\ref{cases})_{3,4}$ and \eqref{ac}, we get
\begin{equation}\label{cc1}
    \begin{split}
         \varphi^*(x,|\tx{b}|)&= \sup_{\xi \in \R^{N\times \tx{m}}} \left \{\langle \tx{b},\xi \rangle - \varphi(x,|\xi|) \right \} \leq \sup_{\xi \in \R^{N\times \tx{m}}} \left \{\langle \tx{b},\xi \rangle - |x_{\tx{m}}|^{\alpha}|\xi|^q\right \} 
         \\
         &\leq(|x_{\tx m}|^\alpha\,|z|^q)^* \lesssim_{\tx{m},q} |x_{\tx{m}}|^{-\frac{\alpha}{q-1}}\tx{b}^{q'}.
    \end{split}
\end{equation}
Therefore, using \eqref{cc1}, $(\ref{cases})_3$, \cite[Lemma 6]{BDS}, \eqref{p0} and \eqref{qp0}, we obtain
\begin{align} \label{cc3}
    \int_{Q_\mm} \varphi^*(x,|\tx{b}|) \dx  &  \leq c\int_{Q_\mm}|x_{\tx{m}}|^{-\frac{\alpha}{q-1}}|\tx{b}|^{q'} \dx \notag \\ & \leq c\int_{Q_\mm} \mathbb{I}_{\{ \dist(\bar{x},\mathcal{C}^{\tx{m}-1}_\lambda)\leq |x_{\tx{m}}|/2\}}|x_{\tx{m}}|^{\frac{-\alpha + q(1-p_0)}{q-1}} \dx \notag \\ & \leq c\int_0^1 t^{\frac{-\alpha + q(1-p_0)}{q-1}} \mathscr{H}^{\tx{m}-1}\left ( \dist(\cdot, \mathcal{C}^{\tx{m}-1}_\lambda)  \leq t/2\right ) \dt \notag \\ & \leq c\int_0^1 t^{ \frac{1-\alpha  -p_0}{q-1}} \dt \leq c_1(\tx{m},N,q,\alpha,p_0) < \infty.
\end{align}
We point out that for any constant $m \geq 1$, by $(\ref{cases})_{1,2}$, it holds
\begin{align*}
    \int_{Q_\mm} \vp(x,m|Du^{\tx{S}}_*|) \dx  =  m^p \int_{Q_\mm}|Du^{\tx{S}}_*|^p \dx =: m^p c_2(\tx{m},N,p) <\infty,
\end{align*}
and, for any $\sigma \geq 1$, by \eqref{cc1} and \eqref{cc3}, we have
\begin{align} \label{laut}
     \int_{Q_\mm}\varphi^*(x,\sigma |\tx{b}|) \dx   \leq \sigma^{q'} \int_{Q_\mm} |x_{\tx{m}}|^{-\frac{\alpha}{q-1}}|\tx{b}|^{q'} \dx \leq  \sigma^{q'}c_1(\tx{m},N,q,\alpha,p_0) <\infty.
\end{align}
Now, for $m_* \geq 1$, we define 
$$ u_0^{\tx{S}}:= m_* u_*^{\tx{S}} \in  C^{\infty}(\overline{Q}_{\mm}\setminus \mathcal{C},\tx{S}^{N-1}_{m_*}), \quad \tilde{u}_0^{\tx{S}}:= m_* \tilde{u}^{\tx{S}} \in C^{\infty}(\overline{Q}_{\mm}, \tx{S}^{N-1}_{m_*}).$$
 Observe that, since $\varphi(x,|Du_0^{\tx{S}}|)=|Du_0^{\tx{S}}|^p$ by \eqref{cases}$_{2}$, for any $\sigma_* \geq 1$ it holds
$$\int_{Q_\mm}\varphi(x,|Du_0^{\tx{S}}|) \dx + \int_{Q_\mm} \varphi^*(x,\sigma_*|\tx{b}|) \dx \leq m_*^pc_2+ \sigma_*^{q'}c_1 \leq m_*^{p_0} (m^{p-p_0}  )+\sigma_*^{p_0'} ( \sigma_*^{q'-p_0'} c_1 ).$$
Recalling that $p<p_0, q'<p_0'$, and taking $\sigma_*=m^{p_0-1}_*$, we can choose $m_*$ so large that
$$   m_*^{p_0} (m^{p-p_0}  )+\sigma_*^{p_0'} ( \sigma_*^{q'-p_0'} c_1 ) < \frac{m_* \sigma_*}{2\sqrt{\frac{1}{4}+N-1}},$$
hence we obtain 
\eqn{by}
$$\int_{Q_\mm}\varphi(x,|Du_0^{\tx{S}}|) \dx + \int_{Q_\mm} \varphi^*(x,\sigma_*|\tx{b}|) \dx < \frac{m_* \sigma_*}{2\sqrt{\frac{1}{4}+N-1}}.$$
We note that by construction
\eqn{db}
$$\tilde{u}_0^{\tx{S}} =  u_0^{\tx{S}} \text{ in } \overline{Q}_{\mm}\setminus (-5/6,5/6)^{\tx{m}}.$$
Since $\tilde{u}_0^{\tx{S}} \in C^{\infty}(\overline{Q}_{\mm}, \tx{S}^{N-1}_{m_*})$, for any $0<\tau<\sigma_* m_*/8(1/4 + N-1)^{1/2}$ we can find $v_\tau \in  C^{\infty}_{\tilde{u}_0^{\tx{S}}}(\overline{Q}_{\mm},\tx{S}^{N-1}_{m_*})$ such that
\begin{align} \label{ff}
    \inf_{w \in  C^{\infty}_{\tilde{u}_0^{\tx{S}}}(\overline{Q}_{\mm},\tx{S}^{N-1}_{m_*})} \int_{Q_\mm} \varphi(x,|Dw|) \dx & \mathrel{\eqmathbox{\overset{\mathrm{}}{>}}} \int_{Q_\mm} \varphi(x,|Dv_\tau|) \dx- \tau \notag \\ & \mathrel{\eqmathbox{\overset{\mathrm{}}{\geq}}} \sigma_* \int_{Q_\mm} \langle Dv_{\tau}, \tx{b} \rangle \dx - \int_{Q_\mm} \varphi^*(x,\sigma_*|
    \tx{b}|) \dx -\tau \notag \\ & \mathrel{\eqmathbox{\overset{\mathrm{\eqref{by}}}{>}}}\sigma_* \int_{Q_\mm} \langle D(v_{\tau}-\tilde{u}_0^{\tx{S}}),\tx{b}\rangle \dx + \sigma_*\int_{Q_\mm} \langle D\tilde{u}_0^{\tx{S}},\tx{b} \rangle \dx \notag \\ & \qquad +  \int_{Q_\mm} \varphi(x,|Du_0^{\tx{S}}|) \dx - \frac{\sigma_* m_*}{2\sqrt{\frac{1}{4}+N-1}} - \tau \notag \\ & \mathrel{\eqmathbox{\overset{\mathrm{}}{=}}} \sigma_* m_* \int_{Q_\mm} \langle D\tilde{u}^{\tx{S}}, \tx{b} \rangle \dx + \int_{Q_\mm} \varphi(x,|Du_0^{\tx{S}}|) \dx - \frac{\sigma_* m_*}{2\sqrt{\frac{1}{4}+N-1}} - \tau \notag \\ & \mathrel{\eqmathbox{\overset{\mathrm{}}{>}}} \frac{3\sigma_* m_*}{8\sqrt{\frac{1}{4}+N-1}} + \int_{Q_\mm} \varphi(x,|Du_0^{\tx{S}}|) \dx\notag \\ & \mathrel{\eqmathbox{\overset{\mathrm{\eqref{db}}}{\geq}}} \frac{3}{8\sqrt{\frac{1}{4}+N-1}} + \inf_{w \in  W^{1,\varphi}_{ \tilde{u}_0^{\tx{S}}}(Q_\mm,\tx{S}^{N-1}_{m_*})} \int_{Q_\mm} \varphi(x,|Dw|) \dx \notag \\ & \mathrel{\eqmathbox{\overset{\mathrm{}}{>}}}  \inf_{w \in  W^{1,\varphi}_{ \tilde{u}_0^{\tx{S}}}(Q_\mm,\tx{S}^{N-1}_{m_*})} \int_{Q_\mm} \varphi(x,|Dw|) \dx;
\end{align}
above, in the second inequality we used the definition of $\varphi^*$ in \eqref{young:conj}, in the first equality we 
used the definition of $u^{\tx{S}}_0=m_*u_*^{\tx{S}}$, and the fact that $\int_{Q_\mm} \langle D(v_{\tau}-\tilde{u}_0^{\tx{S}}),\tx{b}\rangle \dx=0$ due to \eqref{zero:b}; finally, in the fourth inequality, we used  that 
\begin{equation}\label{fourth}
    \int_{Q_\mm} \langle D\tilde{u}^{\tx{S}},\tx{b}\rangle \dx=\frac{1}{\sqrt{\tfrac{1}{4}+N-1}}.
\end{equation}
 Indeed, by \cite[proof of Proposition 19]{BDS},  we can observe that $\tx{b}=0$ except on $\{x_\tx{m}= \pm 1\} \cap \partial Q_\mm$; therefore, denoting by $\nu$ the outward normal of $\partial Q_\mm$, we have
\begin{align} \label{p19}
\int_{\partial Q_\mm} (\tx{b}\, \nu) \cdot u_*^{\tx{S}} \, \d S 
&= \int_{\{x_\tx{m}=1\}} \sum_{i=1}^{N} \bigl( \tx{b}(\bar{x},1) e_\tx{m} \bigr)_i \; (u_*^{\tx{S}})_i(\bar{x},1) \, \d\bar{x} \notag \\
&\quad + \int_{\{x_\tx{m}=-1\}} \sum_{i=1}^{N} \bigl( \tx{b}(\bar{x},-1)  (-e_\tx{m}) \bigr)_i \; (u_*^{\tx{S}})_i(\bar{x},-1) \, \d\bar{x}.
\end{align}
On $\{x_\tx{m}=1\}$, we have $u_*=1/2$, hence recalling \eqref{def:chiamo}
\eqn{1p19}
$$
(u_*^{\tx{S}})_i(\bar{x},1)=\frac{\frac12\delta_{i1}+\sum_{j=2}^{N}\delta_{ij}}{\sqrt{\frac14+(N-1)}},
\qquad
\bigl(\tx{b}(\bar{x},1)\,e_\tx{m}\bigr)_i = b_{\tx{m}}(\bar{x},1)\ \text{for every } i,
$$
and on $\{x_\tx{m}=-1\}$, we have $u_*=-1/2$, so that
\eqn{2p19}
$$
(u_*^{\tx{S}})_i(\bar{x},-1)=\frac{-\frac12\delta_{i1}+\sum_{j=2}^{N}\delta_{ij}}{\sqrt{\frac14+(N-1)}},
\qquad
\bigl(\tx{b}(\bar{x},-1)\,(-e_\tx{m})\bigr)_i = -b_{\tx{m}}(\bar{x},1)\text{ for every $i$,}
$$
where $b=\{b_j\}_{j=1,\dots,\tx m}$ is the vector field of \cite[Definition 9]{BDS}.
Substituting \eqref{1p19} and \eqref{2p19} into \eqref{p19}, we get
\begin{align*}
\int_{\partial Q_\mm} (\tx{b}  \nu) \cdot u_*^{\tx{S}} \, \d S
&= \frac{1}{\sqrt{\frac14+(N-1)}} \int_{\{x_\tx{m}=1\}} b_{\tx{m}}(\bar{x},1)
      \Bigl[ \bigl(\tfrac12+N-1\bigr)+\bigl(\tfrac12-N+1\bigr) \Bigr] \, \d\bar{x} \\
&= \frac{1}{\sqrt{\frac14+(N-1)}} \int_{\{x_\tx{m}=1\}} b_{\tx{m}}(\bar{x},1) \, \d\bar{x},
\end{align*}
 and by \cite[Proposition 19]{BDS}, the last integral equals $1$; therefore by the divergence theorem and recalling $\tilde{u}^{\tx{S}}=u_*^{\tx{S}}$ on $\partial Q_\mm$, we have
\[
\int_{Q_\mm}\langle D\tilde{u}^{\tx{S}},\tx b\rangle \dx=\int_{\partial Q_\mm} (\tx{b} \nu) \cdot u_*^{\tx{S}} \, \d S
= \frac{1}{\sqrt{\frac14+(N-1)}},
\]
and this proves \eqref{fourth}. Therefore, \eqref{ff} shows the presence of the Lavrentiev phenomenon, that is \eqref{lavrentiev} in Theorem \ref{tcontr}.\\

\noindent Next, by direct methods, there exists $v  \in W^{1,\varphi}_{\tilde{u}_0^{\tx{S}}}(Q_\mm,\tx{S}^{N-1}_{m_*})$ such that
\eqn{fa}
$$ \inf_{w \in  W^{1,\varphi}_{ \tilde{u}_0^{\tx{S}}}(Q_\mm,\tx{S}^{N-1}_{m_*})} \int_{Q_\mm} \varphi(x,|Dw|) \dx = \int_{Q_\mm} \varphi(x,|Dv|) \dx.$$
We aim show that $v$ cannot be approximated in $W^{1,\varphi}$-norm by maps in $C^{\infty}_{\tilde{u}_0^{\tx{S}}}(\overline{Q}_{\mm}, \tx{S}^{N-1}_{m_*})$. Assume by contradiction that there exists a sequence $\{ v_{\ell}\}_{\ell} \subset  C^{\infty}_{\tilde{u}_0^{\tx{S}}}(\overline{Q}_{\mm},\tx{S}^{N-1}_{m_*})$ such that $v_{\ell} \to v$ in $W^{1,\varphi}(Q_\mm)$. Then, up to subsequences, by \eqref{fa}  it holds
\begin{align*}
    \inf_{w \in  C^{\infty}_{\tilde{u}_0^{\tx{S}}}(\overline{Q}_{\mm},\tx{S}^{N-1}_{m_*})} \int_{Q_\mm} \varphi(x,|Dw|) \dx & \leq \int_{Q_\mm} \varphi(x,|Dv_{\ell}|) \dx \\ &\to \int_{Q_\mm} \varphi(x,|Dv|) \dx \\ & = \inf_{w \in  W^{1,\varphi}_{ \tilde{u}_0^{\tx{S}}}(Q_\mm,\tx{S}^{N-1}_{m_*})} \int_{Q_\mm} \varphi(x,|Dw|) \dx,
\end{align*}
which  contradicts \eqref{ff}. Taking $\Lambda = m_*$, $\tilde{u}^{\tx{S}}_0 = u_0$ and $v=\bar{u}$, the proof of Theorem \ref{tcontr} in the case $p_0<\tx{m}$ is concluded.
Let us now analyze the other two cases. \\ \\ \noindent
Let $p_0=\tx{m}$. In this case, we take $\mathcal{C}=0$, and consider $\theta_a \in C^{\infty}_c(0,\infty)$ be such that $\mathbb{I}_{(2,\infty)} \leq \theta_a \leq \mathbb{I}_{(1/2,\infty)}$ and $\nr{\theta_a'}_{\infty} \leq 6$. We make the following changes in \eqref{ztilde} and \eqref{ffs}:
$$ \tilde{\tx{Z}}:= \sum_{i=1}^N \tx{Z}_{\tx{m}} \otimes e_i, \quad u_*(x)= \frac{1}{2}\text{sgn}(x_{\tx{m}})\theta\left ( \frac{|x_{\tx{m}}|}{|\bar{x}|}\right), \quad a(x):= |x_{\tx{m}}|^{\alpha} \theta_a(x) \left ( \frac{|x_{\tx{m}}|}{|\bar{x}|} \right ). $$
Then, $(\ref{cases})_{2,4}$ still holds and by \cite[Proposition 15]{BDS} we have $|Du_*^{\tx{S}}|\lesssim_{\tx{m},N}|x_\tx{m}|^{-1}\mathbb{I}_{\{ 2|x_\tx{m}|\leq |\bar{x}| \leq 4|x_\tx{m}|\}}$ and $|\tx{b}| \lesssim_{\tx{m},N} |x_\tx{m}|^{1-\tx{m}}\mathbb{I}_{\{2|\bar{x}|\leq |x_\tx{m}| \leq 4|\bar{x}| \}}$, with $Du_*^\tx{S} \in L^{\tx{m},\infty}(Q_\mm, \R^{N \times \tx{m}})$ and $\tx{b} \in L^{\tx{m}',\infty}(Q_\mm, \R^{N \times \tx{m}})$. Moreover, \eqref{cc1} holds and, using again \cite[Lemma 6]{BDS}, we obtain \eqref{cc3}. The proof then follows exactly as in the previous case. \\ \\
For $p_0>\tx{m}$, let $\mathcal{C}:=\{ 0\}^{\tx{m}-1} \times \mathcal{C}_\lambda$ and $\dim(\mathcal{C})= \frac{\log 2}{\log(1/\lambda)}$, where $\lambda \in  ( 0, 1/2 )$ is chosen such that 
\eqn{p000}
$$p_0 = \frac{\tx{m}-\dim(\mathcal{C})}{1-\dim(\mathcal{C})}.$$ 
\noindent Let $\rho \in C^{\infty}(\R^{\tx{m}}\setminus \mathcal{C})$ be such that $\mathbb{I}_{\{\dist(x_\tx{m}, \mathcal{C}_\lambda)\leq 2|\bar{x}| \}}\leq \rho \leq \mathbb{I}_{\{\dist(x_\tx{m},\mathcal{C}_\lambda) \leq 4|\bar{x}| \}} $, $|D\rho| \lesssim_{\tx{m}} |\bar{x}|^{-1}\mathbb{I}_{\lbrace2|\bar{x}|\leq \dist(x_\tx{m}, \mathcal{C}_\lambda) \leq 4|\bar{x}|\rbrace}$ and let $\rho_a \in C^\infty(\R^\tx{m} \setminus \mathcal{C})$ be such that $\mathbb{I}_{\{\dist(x_{\tx{m}}, \mathcal{C_\lambda}) \leq |\bar{x}|/2 \}} \leq \rho_a \leq \mathbb{I}_{\{ \dist(x_{\tx{m}}, \mathcal{C}_\lambda) \leq 2|\bar{x}|\}}$ and 
\newline $|D \rho_a| \lesssim_{\tx{m}} |\bar{x}|^{-1}\mathbb{I}_{\{ |\bar{x}|/2 \leq \dist(x_{\tx{m}},\mathcal{C}_\lambda) \leq 2|\bar{x}|\}}$, see \cite[Lemma 5]{BDS}, and define
$$ \tilde{\tx{Z}}(x)= \sum_{i=1}^N  \left (\frac{|\bar{x}|^{1-\tx{m}}}{\mathscr{H}^{\tx{m}-1}(\partial \tx{B}_1)}\left[\begin{matrix}
0 & -\bar{x}^t \\
\bar{x} & 0 
\end{matrix} \right ]\rho(x) \right ) \otimes e_i, \quad u_*(x) = (\delta_0^{\tx{m}-1}\times     \mu_\lambda) \ast \left (\frac{1}{2}\text{sgn}(x_{\tx{m}})\theta\left ( \frac{|x_{\tx{m}}|}{|\bar{x}|}\right) \right ),$$ $$ a(x):= |\bar{x}|^{\alpha}(1-\rho_a)(x).$$
The other functions in \eqref{ztilde} and \eqref{ffs} are defined in the same way. Then, $(\ref{cases})_2$ holds, while $\{x \in Q_\mm: |\tx{b}| > 0\}\subset \{x \in Q_\mm: a(x)=|\bar{x}|^{\alpha}\}$, and by \cite[Proposition 15]{BDS}, $|Du_*^{\tx{S}}|\lesssim_{\tx{m},N} |\bar{x}|^{\dim(\mathcal{C})-1}\mathbb{I}_{\{ \dist(x_\tx{m}, \mathcal{C}_\lambda) \leq |\bar{x}|/2\}}$ and $|\tx{b}| \lesssim_{\tx{m},N} |\bar{x}|^{1-\tx{m}}\mathbb{I}_{\{2|\bar{x}|\leq \dist(x_\tx{m},\mathcal{C}_\lambda)\leq 4|\bar{x}|\}}$, with $Du_*^\tx{S} \in L^{\tx{m},\infty}(Q_\mm, \R^{N \times \tx{m}})$ and $\tx{b} \in L^{\tx{m}',\infty}(Q_\mm, \R^{N \times \tx{m}})$. In this case, we have
$$\vp^*(x,|\tx{b}|) \leq \frac{1}{q'}|\bar{x}|^{-\frac{\alpha}{q-1}}|\tx{b}|^{q'}.$$
Applying \cite[Lemma 6]{BDS} we get $\vp^*(\cdot,|\tx{b}|) \in L^1(Q_\mm)$. Moreover, we observe that \eqref{laut} still holds, i.e., $ \int_{Q_\mm}\vp^*(\cdot,\sigma |\tx{b}|) \leq c_1 \sigma^{q'} < \infty$.  The rest of the calculations follows as in the sub-dimensional case. This concludes the proof of Theorem \ref{tcontr}.

\bigskip{}{}

 \par\noindent {\bf Acknowledgments.} C.A. Antonini is a postdoctoral fellow of the National Institute for Advanced Mathematics (INdAM) at the University of Florence.
F. De Filippis is a postdoctoral fellow at the University of Salzburg.
Part of this work was carried out while both authors were postdoctoral fellows at the University of Parma. Part of this work was conducted during C. Pacchiano Camacho’s postdoctoral appointment at UNAM, funded by the Simons Foundation.

\bigskip{}{}

 \par\noindent {\bf Data availability statement.} Data sharing not applicable to this article as no datasets were generated or analysed during the current study.
 \addtocontents{toc}{\protect\setcounter{tocdepth}{-1}}
\section*{Compliance with Ethical Standards}
\label{conflicts}
\addtocontents{toc}{\protect\setcounter{tocdepth}{2}}
\smallskip
\par\noindent
{\bf Funding}.
\\ (i) C.A. Antonini was partly funded by GNAMPA   of the Italian INdAM - National Institute of High Mathematics (grant number not available);
\\ (ii) F. De Filippis has been partially supported through the INdAM - GNAMPA Project (CUP E5324001950001);
\\(iii)  C. Pacchiano Camacho was supported by a grant from Simons Foundation International SFI-MPS-T-Institutes-00011977 JS;  \\ (iv) This research was funded in whole or in part by the Austrian Science Fund (FWF) [10.55776/PAT1850524]. For open access purposes, the author has applied a CC BY public copyright license to any author accepted manuscript version arising from this submission; \\ (v) This work is supported by the European Research Council (ERC) under the European Union’s Horizon 2020 research and innovation programme (grant agreement No. 101220121).

\bigskip
\par\noindent
{\bf Conflict of Interest}. The authors declare that they have no conflict of interest.

\end{document}